\documentclass[11pt]{article}
\usepackage{amssymb, amsmath}
\usepackage{mathtools}

\usepackage[margin=1.2in]{geometry}

\numberwithin{equation}{section}

\usepackage{mathptmx}
\usepackage{authblk}

\newtheorem{theorem}{Theorem}[section]
\newtheorem{lemma}{Lemma}[section]

\newenvironment{proof}[1][Proof]{\begin{trivlist}
\item[\hskip \labelsep {\bfseries #1}]}{\end{trivlist}}

\newenvironment{remark}[1][Remark]{\begin{trivlist}
\item[\hskip \labelsep {\bfseries #1}]}{\end{trivlist}}

\newcommand{\qed}{\nobreak \ifvmode \relax \else
      \ifdim\lastskip<1.5em \hskip-\lastskip
      \hskip1.5em plus0em minus0.5em \fi \nobreak
      \vrule height0.75em width0.5em depth0.25em\fi}

\renewcommand{\Re}{\operatorname{Re}}

\usepackage{xcolor}
\usepackage{hyperref}

\providecommand{\keywords}[1]
{
  \small	
  \textbf{\textit{Keywords: }} #1
}

\providecommand{\MSC}[1]
{
  \small	
  \textbf{\textit{Mathematics Subject Classification 2020: }} #1
}

\begin{document}
%%\title{Conjectured Law of GMC on the Riemann Sphere}
\title{A Family of Probability Distributions Consistent with the DOZZ Formula: Towards a Conjecture for the Law of 2D GMC} 

\author{Dmitry Ostrovsky}
%\affil{{\small 195 Idlewood Drive \\ Stamford, CT 06905, USA \\ dm\_ostrov@aya.yale.edu}}

%%\date{August 5, 2020} 
%%\date{February 5, 2021}
\date{}

\maketitle
\noindent

\begin{abstract}
\noindent
A three parameter family of probability distributions is constructed such that its Mellin transform is defined over the same
domain as the 2D GMC on the Riemann sphere with three insertion points $(\alpha_1,\alpha_2,\alpha_3)$ and satisfies the DOZZ
formula in the sense of Kupiainen \emph{et al}. ({\it Ann. Math.} {\bf 191} (2020) 81 -- 166). The probability distributions in the family 
are defined as products of independent Fyodorov-Bouchaud and powers of Barnes beta distributions of types $(2, 1)$ and $(2, 2).$ 
In the special case of $\alpha_1+\alpha_2+\alpha_3=2Q$ the constructed probability distribution is shown to be consistent with the known small deviation asymptotic of the 2D GMC laws with everywhere positive curvature.
%%The constructed family provides the first step towards a conjecture for the law of the GMC. 
%%The constructed family provides a metric-independent conjecture for the law of the GMC. 
%%It is conjectured that these distributions give the law of the GMC for some background metric. 
\end{abstract}

{\keywords{Gaussian Multiplicative Chaos, DOZZ formula, analytic continuation, Barnes beta probability distributions, Mellin transform, double gamma function, infinite divisibility.}}

{\MSC{60E07, 	60E10, 60D99 (primary), 81T40, 81T20 (secondary).}}

\section{Introduction}
In this paper we contribute to the study of integrability of 2D GMC (Gaussian Multiplicative Chaos) measures on the Riemann sphere. The study of GMC measures, cf. \cite{K2}, \cite{RVrevis}, is a flourishing area of research at the intersection of probability \cite{B}, 
\cite{JS}, \cite{RV2}, \cite{Shamov},  statistical physics of random energy landscapes \cite{Cao2}, \cite{FyoBou}, \cite{FLD}, \cite{FLDR}, and Liouville conformal field theory \cite{Jones}, \cite{BenSch}, \cite{David}, \cite{DS}, \cite{KRV}, \cite{RV1}. 

A fundamental open problem in the theory of GMC is to calculate the distribution of the total mass of the chaos measure. This field
was pioneered by \cite{FyoBou}, \cite{FLDR} and independently by \cite{Me3}, \cite{Me4}, \cite{Me16}, who made precise conjectures
about the Mellin transform of the total mass distribution with and without insertion points on simple 1D shapes such as circle or interval. These conjectures were responsible for the continued development of the field and led to many applications in statistical physics, cf.  \cite{FLD}, and, on the mathematical side, these conjectures led to the creation of the theory of Barnes beta distributions, cf. \cite{MeIMRN}, \cite{Me13}, \cite{Me14}. We refer the interested reader to \cite{Me18} for a comprehensive review of all of these developments. 

The conjectures about the total mass in 1D were facilitated by the knowledge of the integer moments of the GMC chaos measure in question,
which were represented by the Selberg integral on the interval and the Morris integral on the circle, cf. \cite{ForresterBook}, chapter 4. The task of formulating a conjecture
was then tantamount to analytically continuing the corresponding integral as a function of its dimension, \emph{i.e.} the order of the integer moment of the measure, to the complex plane, thereby conjecturing the Mellin transform of the total mass. Such a procedure does not guarantee uniqueness and is not mathematically rigorous, especially as the positive integer moments of a GMC measure become infinite at a sufficiently high order. As a result, the problem of rigorously computing the law of the total mass remained out of reach until 2018. A breakthrough was made in \cite{David}, where the connection between GMC and Liouville conformal field theory was established. This connection 
along with the machinery developed in \cite{KRVlocal} and \cite{KRV} led
to the proofs of: the conjecture of \cite{FyoBou} about the law of the GMC on the circle without insertion points in \cite{Remy}, 
the conjectures of \cite{FLDR} and \cite{Me4}, \cite{MeIMRN} about the law the GMC on the interval with two insertion points in \cite{RemyZhu}, and, most recently, the conjecture of \cite{Me16} about the law of GMC on the circle with a single insertion point in
\cite{RemyZhu2}, see also \cite{Nuj} and \cite{NujCh} for a different approach based on orthogonal polynomials. It needs to be emphasized that these proofs 
were the proofs of uniqueness, which was lacking in the analytic continuation approach, and so required the conjectures as an essential starting point.
%% as they provided constraints on the analytic functions involved and thereby established uniqueness. 

%%the conjectures themselves played a central role in these proofs for they pro

The study of integrability of 2D GMC measures goes back to the pioneering work of \cite{Dorn} and \cite{ZZ}, which conjectured
3-point correlation functions of the Liouville conformal field theory that is known as the celebrated DOZZ formula, cf.  \cite{Nak} and \cite{Rib}  for review. Their computations were based on the complex Selberg integral that was evaluated independently in \cite{Aom} and \cite{DF}, cf.
\cite{ForresterWarnaar} for review. The DOZZ formula remained without a rigorous mathematical footing until the work of \cite{David} and \cite{KRV}, which established the connection between the GMC and Liouville theories and formulated and proved the DOZZ formula, respectively. In the modern mathematical language, the DOZZ formula gives the value
of the Mellin transform of the GMC measure with three insertion points $(\alpha_1, \alpha_2, \alpha_3)$ on the Riemann sphere at the point
\begin{align}
s_0 = &\frac{\alpha_1+\alpha_2+\alpha_3-2Q}{\gamma},
\end{align}
where $\alpha_i$ are the insertion points, $\gamma$ is the fundamental constant of the GMC\footnote{The GMC measure in 2D is classified as subcritical if $\gamma<2,$ critical if $\gamma=2,$ and supercritical if $\gamma>2.$ The constant $\gamma^2$ is referred to as the intermittency parameter of the GMC. The constant $Q$ is related to $\gamma$ by $Q=\frac{2}{\gamma} + \frac{\gamma}{2}.$}, 
and $Q$ is related to the central charge $c_L$ of the Liouville theory by $c_L=1+6Q^2,$ cf. \cite{Vargas} for a review of the original and the modern mathematically rigorous approaches. The GMC measure on the sphere depends non-trivially on the choice of
the conformal background metric. The proof of the DOZZ formula in \cite{KRV} was given for the metric $g(x)=|x|_+^{-4}.$ The remarkable feature of the DOZZ formula is that its right-hand side, which is the value of the Mellin transform at $s_0,$ the so-called structure constant, remains the same, up to a metric-dependent scaling factor, for any background metric \cite{David}. 

The focus of this paper is the metric-independent (universal) part of the structure constant. As the dependence of the GMC law on the choice of the metric is unknown, it is natural to ask about what can be learned about this law from the structure constant
itself. Given the universality of the structure constant, this information is equally universal, \emph{i.e.} applies to any GMC measure on the sphere.  Clearly, as the DOZZ formula holds for any conformal background metric and the GMC measure depends on the metric, the DOZZ formula does not determine the GMC law uniquely. This lack of uniqueness is stronger than the lack of uniqueness of 1D GMC conjectures, where the match was made at all finite integer moments as opposed to a single point. Nonetheless, the DOZZ formula puts a highly non-trivial constraint on the GMC law so that by constructing probability distributions that are consistent with the DOZZ formula one can substantially narrow down the search for this law. %Further, any such distribution provides a metric-independent conjecture for this law. 
%%It is not possible to identify the specific GMC measure for which such a conjecture is made because such the distribution is metric-independent by construction. 

The main contribution of this paper is the construction of a family
of positive probability distributions, whose Mellin transform is defined and analytic over the same domain as that of the GMC measure with three insertion points and takes on the same value at $s_0$ as a function of $\gamma$ and insertion points as the Mellin transform of the GMC measure, \emph{i.e.} satisfies the DOZZ formula. Thus, our result can be thought of as analytic continuation of the DOZZ formula as a function of $s_0$ to the complex plane. The idea of such continuation and mathematical methods used to effect it are similar to our work in 1D and are based on the theory of Barnes beta probability distributions. We first identify the minimal solution, whose structure is essentially imposed uniquely by 
the structure of the Upsilon terms in the DOZZ formula itself. We prove that it is the Mellin transform of a probability distribution
by factoring it into the product of Mellin transforms of Barnes beta distributions of types  $(2,1)$ and $(2,2).$ %This is the main result of the paper. 
The minimal solution has no free parameters and is symmetric in $(\alpha_1, \alpha_2, \alpha_3).$ We then deform 
each of the Barnes beta factors so that they keep the same value at $s_0$ and remain the Mellin transform of a probability distribution.
This produces a three parameter family of probability distributions that extend the minimal solution. The three deformation parameters
control the symmetry of the distribution in $(\alpha_1, \alpha_2, \alpha_3)$ and allow for solutions that are symmetric in
$(\alpha_1, \alpha_2, \alpha_3)$ or only symmetric in $(\alpha_1, \alpha_3)$ or not symmetric at all. 

The metric-dependent scaling factor is the only information about the metric dependence that the DOZZ formula reveals about the GMC law. We provide the trivial analytic continuation of this factor that is also consistent with the known scaling behavior of the GMC measure. Otherwise, we do not investigate the dependence on the metric any further as it requires additional information about the GMC law that is currently unknown such as its first moment as a function of the insertion points, for example. In particular, we note that our construction does not match the first moment of the GMC law at zero insertion points or the small deviation asymptotic 
of GMC measures having everywhere positive curvature. However, in the special case of $s_0=0$ $(\alpha_1+\alpha_2+\alpha_3=2Q)$ %when the metric-dependent scaling factor is identically one, 
we provide a modified construction that matches the small deviation asymptotic as well as the DOZZ formula. 
%%the minimal solution so that it keeps the same value at $s_0$ and remains the Mellin transform of 

The main technical innovation is the construction of one-parameter deformations of general Barnes beta
distributions of types $(2,1)$ and $(2,2)$ having the property that the deformed Mellin transform has the same value at a given point
as the original Mellin transform. We also provide an original construction of conformal metrics on the sphere having the property that the corresponding GMC measure is symmetric in all the insertion points. 
%Also, our construction provides the first known use of the Barnes beta distribution of type $(2,1)$
%in the context of GMC. Up to now, all the known laws of 1D GMC measures were expressed in terms of products of type $(2,2)$ 
%distributions and Frechet and lognormal factors so that the appearance of type $(2,1)$ distributions here is a new feature of 2D GMC. 
Finally, we give an explicit computation of the constant $\Upsilon'_{\frac{\gamma}{2}}(0)$ in the DOZZ formula, which appears to be new. 

As explained above, all major advances of GMC integrability research were preceded by exact conjectures and, further, the 
Liouville theory based method of proof specifically requires a conjecture as its starting point. This makes us believe that our construction of a family of probability distributions satisfying the DOZZ formula provides an essential advance of the field and opens up an avenue for new integrability results in both the 2D GMC and Liouville theories as it provides the first necessary step towards formulating a conjecture for the law of GMC on the sphere. The structure constant in the DOZZ formula is quite complicated and imposes a very delicate constraint on the underlying distribution so that it is reasonable to suppose that by matching it
one is not too far off from the actual law of the GMC, especially as the domain of analyticity of the constructed Mellin transform coincides with that of the Mellin transform of the GMC. Our construction is flexible enough to narrow down the search further by matching
the symmetry of the GMC law in $(\alpha_1, \alpha_3)$ or $(\alpha_1, \alpha_2, \alpha_3)$ that are known to hold for particular background metrics. %While our analytic continuation is not the most general possible that is consistent with the DOZZ formula, we believe that the general solution is likely to be a deformation of our minimal solution. 

The plan of the paper is as follows. In Section 2 we briefly remind the reader of the definition of GMC measures on the Riemann sphere,
state the DOZZ formula in terms of such measures, and summarize the properties of the GMC measures that we match. In Section 3 we summarize the main mathematical tools that are used in the rest of the paper. In Section 4 we give the main results. In Section 5 we give the proofs. Section 6 shows where our construction falls short of being a conjecture, gives a detailed summary of all constraints that need to be satisfied to formulate a conjecture, and lists some open questions. Conclusions are given in Section 7. The Appendix summarizes the scaling and symmetry properties of GMC measures on the sphere and gives
our construction of GMC measures on the sphere that are symmetric in all insertion points.

\section{Problem Statement}
In this section we will briefly review the definition of the 2D GMC on the Riemann sphere (identified with the extended complex plane) corresponding to a background metric $g(x)$ 
and refer the reader to \cite{David}, \cite{KRV}, \cite{RVrevis}, \cite{Vargas} and for details.

Define the following quantities,
\begin{equation}
Q=\frac{2}{\gamma} + \frac{\gamma}{2}, \; 0<\gamma<2.
\end{equation}
\begin{equation}\label{taudef}
\tau = \frac{4}{\gamma^2}.
\end{equation}
\begin{equation}
\Upsilon_{\frac{\gamma}{2}} (x) = \frac{1}{\Gamma_{\frac{\gamma}{2}}(x)\,\Gamma_{\frac{\gamma}{2}}(Q-x)}.\label{ups}
\end{equation}
where $\Gamma_{\frac{\gamma}{2}}(x)$ denotes the physicist's double gamma function, cf. \cite{Nak} and $\Upsilon_{\frac{\gamma}{2}} (x)$ is the Upsilon function.  

Consider the subcritical GMC on the Riemann sphere. Given a conformal metric on the sphere, let the Gaussian Free Field $X_g(x)$ corresponding to the metric be defined by
\begin{equation}\label{Xg}
{\bf E}[X_g(x)\, X_g(y)] = \log\frac{1}{|x-y|} - \frac{1}{4}\log g(x) -  \frac{1}{4}\log g(y) + \chi_g.
\end{equation}
$\chi_g$ is a constant depending on the metric. The field $X_g(x)$ is defined to have the zero average with respect to the curvature of the metric $g(x),$ cf. \eqref{curvatureaverage}. If one writes the metric in the form $g(x)=e^{\varphi(x)}\,g_+(x),$ where $g_+(x)=|x|_+^{-4},$
$|x|_+ = \max(1, |x|),$  then the constant $\chi_g$ is given by
\begin{equation}\label{chigdef}
\chi_g = \frac{1}{32\pi} \Bigl[ \int\limits_{\mathbb{C}}|\nabla \varphi(x)|^2\, dx + 8 \int\limits_{0}^{2\pi} \varphi(e^{i\theta}) d\theta   \Bigr].
\end{equation}
It can be positive, negative or zero.
For example, for the metric  $g(x) = |x|_+^{-4}$ it equals 0. 
Given three insertion points $\alpha_1,\alpha_2, \alpha_3,$
let
\begin{equation}\label{alphabarsum}
\bar{\alpha}=\alpha_1+\alpha_2+\alpha_3.
\end{equation}
Define the subcritical GMC\footnote{As the field $X_g(x)$ is not defined pointwise, the usual renormalization procedure is required to
define the GMC, cf. \cite{RVrevis}.}
 and the corresponding $\rho_g$ by the formulas,
\begin{gather}
\rho_g(\alpha_1, \alpha_2, \alpha_3, \gamma) = e^{\frac{\gamma^2 \chi_g}{2}} \int_\mathbb{C}
\frac{g(x)^{-\frac{\gamma}{4}\bar{\alpha}}}{|x|^{\gamma\alpha_1} |1-x|^{\gamma\alpha_2}} 
M_{\gamma, g}(dx), \label{rhog} \\
M_{\gamma, g}(dx) = e^{\gamma\,X_g(x) -\frac{\gamma^2}{2} {\bf E}[X_g(x)^2]} g(x) dx.
\end{gather}
The natural boundaries for the variables involved that guarantee the existence of the Mellin transform\footnote{Throughout this paper
we define the Mellin transform of a random variable $X$ to mean ${\bf E} [X^s]$ with the exception of the Mellin transform of
$\rho_g(\alpha_1, \alpha_2, \alpha_3, \gamma)$ written as ${\bf E}\bigl[ \rho_g(\alpha_1, \alpha_2, \alpha_3, \gamma)^{-s}\bigr]$ to be consistent with
its definition in \cite{KRV} and \cite{Vargas}.}
${\bf E}\bigl[ \rho_g(\alpha_1, \alpha_2, \alpha_3, \gamma)^{-s}\bigr]$
of the quantity $\rho_g$ are given in \cite{Vargas}, cf. Eq. (2.12):
%% \href{https://arxiv.org/abs/1712.00829}{Vargas}, Eq. (2.12) :
\begin{gather}
-\Re(s)<\tau, \label{mc}\\
-\Re(s)< \frac{2}{\gamma}(Q-\alpha_k), \label{mc2}\\
\alpha_k<Q. \label{bin1}
\end{gather}
%%The Mellin transform of the quantity $\rho_g$ exists under the following conditions,
Let 
\begin{align}
s_0 = &\frac{\bar{\alpha}-2Q}{\gamma},  \label{s0}
\end{align}
and
\begin{equation}
l(z) = \frac{\Gamma(z)}{\Gamma(1-z)}.
\end{equation}
Assume that $s_0$ satisfies (\ref{mc}) and (\ref{mc2}). Then, the main result of \cite{KRV} is the proof\footnote{The proof in \cite{KRV} is restricted to the metric $g(x)=|x|_+^{-4}$ but goes through for a general metric on the sphere, cf. \cite{David}.} of the DOZZ formula %%for the general metric is
\begin{equation}
\Gamma(s) \, {\bf E}\bigl[ \rho_g(\alpha_1, \alpha_2, \alpha_3, \gamma)^{-s} \bigr]\Big\vert_{s=s_0} = \frac{\gamma}{2}\, e^{\frac{s_0^2\gamma^2}{2}\chi_g}\,
\Bigl(  \pi  l(\frac{\gamma^2}{4}) (\frac{\gamma}{2})^{2-\frac{\gamma^2}{2}}\Bigr)^{-s_0}\,
 \frac{\Upsilon'_{\frac{\gamma}{2}}(0)\prod\limits_{i=1}^3 \Upsilon_{\frac{\gamma}{2}} (\alpha_i)}
{ \Upsilon_{\frac{\gamma}{2}} (\frac{\bar{\alpha}}{2}-Q)
\prod\limits_{i=1}^3 \Upsilon_{\frac{\gamma}{2}} (\frac{\bar{\alpha}}{2}-\alpha_i)}.\label{gDOZZ}
\end{equation}
%%\textcolor{red}{Q2. Is this correct? In your notes at the beginning of section 3 this is written without the $\gamma$ in the first fraction $\gamma/2$ on the right-hand side, whereas in (2.1) at the top of page 2 it is written with $\gamma.$ }
It must be emphasized that this result requires $\alpha_i>0$ $\forall i,$ cf. Lemma \ref{Inequalitiies} below, but the Mellin
transform is defined if $\alpha_i=0$ $\forall i,$ cf. \eqref{mc} -- \eqref{bin1} and corresponds to the total mass of the GMC measure. 
The difference between the two is that the DOZZ theorem requires $s_0$ to satisfy \eqref{mc} and \eqref{mc2}, whereas the existence
of the Mellin transform does not.

The law of $\rho_g(\alpha_1, \alpha_2, \alpha_3, \gamma)$ is not known for any metric. Aside from its existence,
$\rho_g(\alpha_1, \alpha_2, \alpha_3, \gamma)$ is known to possess three additional properties.
\begin{enumerate}
\item
When $g(x)$ is multiplied by a positive
constant $\lambda,$  $\rho_g(\alpha_1, \alpha_2, \alpha_3, \gamma)$ satisfies the scaling invariance,  
\begin{equation}
{\bf E}\bigl[ \rho_{\lambda g}(\alpha_1, \alpha_2, \alpha_3, \gamma)^{-s} \bigr] = \lambda^{s s_0 \frac{\gamma^2}{4}} \,{\bf E}\bigl[ \rho_{g}(\alpha_1, \alpha_2, \alpha_3, \gamma)^{-s} \bigr],\label{rhogscaling}
\end{equation}
cf. Appendix A. 
\item 
If the curvature of the metric $R_g(x) = -\frac{1}{g(x)} \Delta \log g(x)$ satisfies $R_g(x)>0$ $\forall x\in\mathbb{C},$ then $\rho_g$ satisfies the following
small deviation estimate,
\begin{equation}\label{smalldeviation}
{\bf P}\Bigl [\rho_g(\alpha_1, \alpha_2, \alpha_3, \gamma) \leq \varepsilon\Bigr] \thicksim e^{-\varepsilon^{-\frac{4}{\gamma^2}}}
.
\end{equation}
This is a corollary of Theorem 4.5 in \cite{Lac} and the identity
\begin{equation}\label{curvatureaverage}
\int_\mathbb{C} X_g(x) R_g(x) g(x) dx = 0
\end{equation}
that is valid for any metric. 
This estimate implies that $\rho_g$ does not have a lognormal factor for all metrics having everywhere positive curvature. For example, this conclusion applies to the natural metric $g(x) = (1+|x|^2)^{-2}.$

\item
If the background metric $g(x)$ satisfies the property 
\begin{equation}\label{ginverseproperty}
g\bigl(\frac{1}{x}\bigr) = |x|^4\,g(x), \,x\in \mathbb{C},
\end{equation}
then, 
\begin{equation}\label{a1a3}
\rho_g(\alpha_3, \alpha_2, \alpha_1, \gamma) = \rho_g(\alpha_1, \alpha_2, \alpha_3, \gamma) 
\end{equation}
in law, cf. Appendix A. The natural conformal metrics on the Riemann sphere such as 
\begin{align}
g(x) = & |x|_+^{-4}, \label{g+} \\
g(x) = & \frac{1}{(1+|x|^2)^2}, \label{ground}
\end{align}
satisfy (\ref{ginverseproperty}) so that the corresponding  $\rho_g$ is symmetric under  $\alpha_1\leftrightarrow \alpha_3.$ 
%%We finally note that the constant $\chi_g$ can be computed explicitly. For example, for the metric  $g(x) = |x|_+^{-4}$ it equals 0. 
In addition, we prove in this paper that for any metric $g(x)$ satisfying \eqref{ginverseproperty} the transformed metric 
\begin{equation}
T[g](x) = g(x)  + \frac{g\bigl(\frac{x}{x-1}\bigr)}{|x-1|^4} + g(1-x)
\end{equation}
has the property that the corresponding $\rho_{T[g]}$ is symmetric in all  $(\alpha_1, \alpha_2, \alpha_3),$ cf. Appendix B.
\end{enumerate}
We finally note that it is easy to write down the multiple integral expression for the $n$th moment of $\rho_g,$
\begin{equation}\label{exactfirstmoment}
{\bf E}[\rho_g(\alpha_1, \alpha_2, \alpha_3, \gamma) ^n] = e^{\frac{n^2\gamma^2}{2}\chi_g} \int\limits_{\mathbb{C}^n} \prod\limits_{i=1}^n 
g(x_i)^{1-\frac{\gamma}{4}\overline{\alpha}-\frac{\gamma^2}{4}(n-1)} |x_i|^{-\gamma\alpha_1}|1-x_i|^{-\gamma\alpha_2}
\prod\limits_{i<j}^n |x_i-x_j|^{-\gamma^2} dx
\end{equation}
but is difficult to compute in closed form for non-zero $\alpha_i,$ even for $n=1.$ In the simplest case of $n=1$ and zero insertion points, this simplifies to
\begin{equation}\label{firstmomentzeroalpha}
{\bf E}[\rho_g(\alpha_i=0\,\forall i, \gamma)] = e^{\frac{\gamma^2}{2}\chi_g} \int\limits_{\mathbb{C}}  g(x)\,dx.
\end{equation}

In what follows we will consider the problem of constructing positive random variables $M(\alpha,\gamma)$ parameterized by 
$(\alpha_1, \alpha_2, \alpha_3),$ $\gamma,$ and possibly some additional free parameters whose Mellin transform ${\bf E}[M^s(\alpha,\gamma)]$: 
\begin{itemize}
\item is defined under the conditions in \eqref{mc} -- \eqref{bin1},
\item satisfies the DOZZ formula (\ref{gDOZZ}) without the metric-dependent scaling factor,
\begin{equation}\label{nogDOZZ}
\Gamma(s) \, {\bf E}\bigl[ M^s(\alpha,\gamma) \bigr]\Big\vert_{s=s_0} = \frac{\gamma}{2}\,
\Bigl(  \pi  l(\frac{\gamma^2}{4}) (\frac{\gamma}{2})^{2-\frac{\gamma^2}{2}}\Bigr)^{-s_0}\,
 \frac{\Upsilon'_{\frac{\gamma}{2}}(0)\prod\limits_{i=1}^3 \Upsilon_{\frac{\gamma}{2}} (\alpha_i)}
{ \Upsilon_{\frac{\gamma}{2}} (\frac{\bar{\alpha}}{2}-Q)
\prod\limits_{i=1}^3 \Upsilon_{\frac{\gamma}{2}} (\frac{\bar{\alpha}}{2}-\alpha_i)},
\end{equation}
\item is flexible enough, depending on the choice of the free parameters, to be symmetric in law in all $(\alpha_1, \alpha_2, \alpha_3)$ or to be
only symmetric under $(\alpha_1\leftrightarrow \alpha_3)$ or to have no particular symmetry. 
%%possess the symmetry in law in \eqref{a1a3}
%%\begin{equation}\label{Msymmetry}
%%M(\alpha_3, \alpha_2, \alpha_1, \gamma)= M(\alpha_1, \alpha_2, \alpha_3, \gamma)
%%\end{equation}
%%or
\end{itemize}
Such a construction provides the necessary first step for formulating a conjecture for the law of $\rho_g(\alpha_1, \alpha_2, \alpha_3, \gamma)$ 
in the form
\begin{equation}\label{naivecon}
\rho_g(\alpha_1, \alpha_2, \alpha_3, \gamma) = e^{-s_0 \frac{\gamma^2}{2}\chi_g} \, M^{-1}(\alpha,\gamma),
\end{equation}
which satisfies the DOZZ formula \eqref{gDOZZ} and is consistent with the scaling invariance \eqref{rhogscaling}, cf. \eqref{lambdascaling} in Appendix A. 
We note that in this proposal the dependence on the metric comes  from both the exponential pre-factor and the free parameters, which in general depend on the metric. Of particular interest is the special case of $\alpha_i=0$ $\forall i,$ which then corresponds to the total mass of the GMC on the sphere
corresponding to the metric $g(x).$ 

The task of formulating an actual conjecture for a particular metric requires some understanding of how the GMC law depends on the choice of the metric and at the very least involves matching the first moment in \eqref{firstmomentzeroalpha} and, if the metric has everywhere positive curvature, the small deviation asymptotic in \eqref{smalldeviation}. These questions are outside the scope of this paper as they require information about the GMC law that goes beyond what is revealed by the DOZZ formula alone. Our construction will need to be refined once this information becomes available
as in its current form it does not satisfy \eqref{smalldeviation} or \eqref{firstmomentzeroalpha} in general. In the special case of $s_0=0$ we provide a modified construction that is consistent with the DOZZ formula and the desired small deviation asymptotic in \eqref{smalldeviation}. %Perhaps not incidentally, in the case of $s_0=0$ the metric-dependent scaling factor in the DOZZ formula is identically one, which obviates the need for its analytic continuation. 

%%The task of formulating an actual conjecture for a particular metric is outside the scope of this paper. Ultimately, this task requires some understanding of how the GMC law depends on the choice of the metric and at the very least involves matching the first moment in \eqref{firstmomentzeroalpha}. 

\section{Building Blocks}

%%We start with the definition of the Upsilon function.

%%\begin{equation}
%%\Upsilon_{\frac{\gamma}{2}} (x) = \frac{1}{\Gamma_{\frac{\gamma}{2}}(x)\,\Gamma_{\frac{\gamma}{2}}(Q-x)}.\label{ups}
%%\end{equation}
We start with the relationship between the physicist's double gamma function that appears in (\ref{ups}) and the double gamma function, $\Gamma_2(x|\tau),$ that is typically used in the mathematical literature and that we will use below. It is given in \cite{RemyZhu}, footnote 7,
%%\href{https://arxiv.org/abs/1804.02942}{Remy-Zhu}, footnote 7,
\begin{equation}\label{diction}
\Gamma_{\frac{\gamma}{2}}(x) = (\frac{2}{\gamma})^{\frac{1}{2} (x-\frac{Q}{2})^2} \frac{\Gamma_2(\frac{2x}{\gamma}|\tau)}{\Gamma_2(\frac{Q}{\gamma}|\tau)}
,
\end{equation}
where the double Gamma $\Gamma_2(x|\tau)$ is reviewed in sect. 3 of \cite{Me18}, see also sect. 2 of \cite{Me14} and \cite{Ruij}. %%\href{https://arxiv.org/abs/1803.06677}{Ostrovsky}.

The first building block that is used below is the Barnes beta distribution of type (2, 2). We denote it by $\beta_{2,2}(b_0, b_1, b_2),$ cf Theorem 7.1 in Sect. 7.1 and Eq. (7.24) in \cite{Me18}. %%\href{https://arxiv.org/abs/1803.06677}{Ostrovsky}.
Given $b_0>0,$ $b_1,$ $b_2,$ such that $b_0+b_1>0,$  $b_0+b_2>0,$ $b_0+b_1+b_2>0,$ and $b_1b_2>0,$\footnote{The conditions 
$b_0>0,$ $b_0+b_1>0,$  $b_0+b_2>0,$ $b_0+b_1+b_2>0,$ are needed for the existence of the right-hand side of (\ref{barnesbeta22}),
$b_1b_2>0$ is needed for the positivity of the L\'evy-Khinchine spectral function.} and $\Re(s) + b_0>0,$  $\Re(s) + b_0+b_1>0,$
 $\Re(s) + b_0+b_2>0,$  $\Re(s) + b_0+b_1+b_2>0,$
\begin{equation}\label{barnesbeta22}
{\bf E}[\beta_{2,2}^{s}] =
\frac{\Gamma_2(s+b_0\,|\,\tau)}{\Gamma_2(b_0\,|\,\tau)}
\frac{\Gamma_2(b_0+b_1\,|\,\tau)}{\Gamma_2(s+b_0+b_1\,|\,\tau)}
\frac{\Gamma_2(b_0+b_2\,|\,\tau)}{\Gamma_2(s+b_0+b_2\,|\,\tau)}
\frac{\Gamma_2(s+b_0+b_1+b_2\,|\,\tau)}{\Gamma_2(b_0+b_1+b_2\,|\,\tau)}.
\end{equation}
The Mellin transform has the asymptotic behavior,
\begin{equation}
{\bf E}[\beta_{2,2}^{s}]  \thicksim s^{-\frac{b_1 b_2}{\tau}}, \; \Re(s)\longrightarrow +\infty.
\end{equation}

The second building block is the Barnes beta distribution of type (2, 1). We denote it by $\beta_{2,1}(b_0, b_1),$ cf Theorem 7.5 in Sect. 7.2 and Eq. (7.35) in \cite{Me18}. %%\href{https://arxiv.org/abs/1803.06677}{Ostrovsky}.
Given $b_0>0$ and $b_1>0,$ 
\begin{equation}\label{barnesbeta21}
{\bf E}[\beta_{2,1}^{s}] =
\frac{\Gamma_2(s+b_0\,|\,\tau)}{\Gamma_2(b_0\,|\,\tau)}
\frac{\Gamma_2(b_0+b_1\,|\,\tau)}{\Gamma_2(s+b_0+b_1\,|\,\tau)}, \, \Re(s)>-b_0.
\end{equation}
The Mellin transform has the asymptotic behavior,
\begin{equation}\label{beta21asym}
{\bf E}[\beta_{2,1}^{s}]  \thicksim e^{\frac{b_1}{\tau} s\log(s) +O(s)}, \; \Re(s)\longrightarrow +\infty.
\end{equation}

We will also need the functional equations of the double gamma function,
\begin{align}
\frac{\Gamma_2(s+\tau|\tau)}{\Gamma_2(s+\tau+1|\tau)} & = \frac{\tau^{\frac{s+\tau}{\tau}-\frac{1}{2}}}{\sqrt{2\pi}} \, \Gamma\bigl(\frac{s+\tau}{\tau}\bigr), \label{doublefunceq1}\\
\frac{\Gamma_2(s|\tau)}{\Gamma_2(s+\tau|\tau)} & = \frac{1}{\sqrt{2\pi}}\,\Gamma(s),  \label{doublefunceq2}
\end{align}
 %%the identity that follows from the functional equations of $\Gamma_2,$ 
cf.  Eqs. (3.5) and (3.6) in \cite{Me18} %%\href{https://arxiv.org/abs/1803.06677}{Ostrovsky} 
with  $a_1=1,$ $a_2=\tau.$ 
%%or the equations \eqref{doublefunceq1} and  \eqref{doublefunceq1} below,
%%which produce 
%%\begin{equation}\label{shift}
%%\frac{\Gamma_2(s|\tau)}{\Gamma_2(s+\tau+1|\tau)} =  \frac{\tau^{\frac{s}{\tau}+\frac{1}{2}}}{2\pi}\, \Gamma(s)  \Gamma\bigl(\frac{s+\tau}{\tau}\bigr).
%%\end{equation}

We finally recall the definition of the Fyodorov-Bouchaud random variable, cf. \cite{FyoBou}.
\begin{equation}
{\bf E}\bigl[ Y^{s} \bigr] = 
\Gamma(1+\frac{s}{\tau}).
\end{equation}
More generally, given $b>0,$ $c>0,$
\begin{equation}
{\bf E}\bigl[ Y(b, c)^{s} \bigr] = \frac{\Gamma(b+\frac{s}{c})}{\Gamma(b)}, 
\end{equation}
is the Mellin transform of a (generalized) Frechet RV. It has the density,
\begin{equation}
\text{pdf}(y) = \frac{c}{\Gamma(b)} \, y^{c b-1}\, e^{-y^c}.
\end{equation}

We end this section with a comment about the small deviation asymptotics of the random variables considered above. For our purposes we need the estimate
of the form ${\bf P}(X^{-1}\leq \varepsilon)$ for $X=\beta_{21}(b_0, b_1), \, X=  Y(b, c),$ and products of such random variables. The tail of $\beta_{22}$ is too thin for our needs. We start with $Y(b, c).$ The probability ${\bf P}(Y(b, c)^{-1}\leq \varepsilon)$ can be computed exactly in terms of incomplete Gamma function but we can estimate it as follows. This probability involves the long tail of $Y(b, c),$ where its probability density is
dominated by the $e^{-y^c}$ term, corresponding to the asymptotic of the Mellin transform of the form 
\begin{equation}
{\bf E}\bigl[ Y(b, c)^{s} \bigr] \thicksim e^{\frac{1}{c}s\log s}
. 
\end{equation}
Hence the small deviation probability behaves like
\begin{equation}\label{smalldeviationgeneral}
{\bf P} (Y(b, c)^{-1}\leq \varepsilon) \thicksim  e^{-\frac{1}{\varepsilon^c}}.
\end{equation}
It remains to notice that the asymptotic behavior of the Mellin transform of $\beta_{21}$ and of products of such variables is the same as that of $Y(b, c),$
cf. \eqref{beta21asym}, \emph{i.e.} 
\begin{equation}\label{mellinasymptoticgeneral}
{\bf E}\bigl[X^{s} \bigr] \thicksim e^{\frac{1}{c}s\log s}
\end{equation}
for some constant $c$ so that \eqref{smalldeviationgeneral} still holds. For example, for $X=\beta_{21}(b_0, b_1)$ we have from \eqref{beta21asym} $c=\tau/b_1$ and for the product $X=\beta_{21}(b_0, b_1)\beta_{21}(b'_0, b'_1)$ we similarly have $c=\tau/(b_1+b'_1).$

\section{Results}
In this section we will construct probability distributions that satisfy (\ref{gDOZZ}) and are defined subject to
(\ref{mc}), (\ref{mc2}), and \eqref{bin1}. We start with the \emph{minimal} solution. Throughout this section
we assume that the $\alpha$s satisfy \eqref{bin1} and $s_0$ in \eqref{s0} satisfies (\ref{mc}), (\ref{mc2}).

Introduce the notation,
\begin{equation}\label{Cconst}
C_\gamma(\alpha_1, \alpha_2, \alpha_3) =  \frac{\gamma}{2}\, \Bigl(  \pi  l(\frac{\gamma^2}{4}) (\frac{\gamma}{2})^{2-\frac{\gamma^2}{2}}\Bigr)^{-s_0}\,
\frac{\Upsilon'_{\frac{\gamma}{2}}(0)\prod\limits_{i=1}^3 \Upsilon_{\frac{\gamma}{2}} (\alpha_i)}
{ \Upsilon_{\frac{\gamma}{2}} (\frac{\bar{\alpha}}{2}-Q)
\prod\limits_{i=1}^3 \Upsilon_{\frac{\gamma}{2}} (\frac{\bar{\alpha}}{2}-\alpha_i)}.
\end{equation}
%%From now we will denote the product of ratios of the Upsilon factors in (\ref{gDOZZ}) by $C_\gamma(\alpha_1, \alpha_2, \alpha_3)$ write this formula in the form,
so that we can write the DOZZ formula in the simplified form,
\begin{equation}
\Gamma(s) \, {\bf E}\bigl[ \rho_g(\alpha_1, \alpha_2, \alpha_3, \gamma)^{-s} \bigr]\Big\vert_{s=s_0} =   e^{\frac{s_0^2\gamma^2}{2}\chi_g}\,
C_\gamma(\alpha_1, \alpha_2, \alpha_3),
\end{equation}
where the constant $C_\gamma(\alpha_1, \alpha_2, \alpha_3)$ is independent of the metric $g(x).$

We start with a heuristic derivation that motivates the formal results. 
Let us observe that the Upsilon terms in the constant  $C_\gamma(\alpha_1, \alpha_2, \alpha_3)$ have a particular structure. 
By expanding it in terms of the double gamma factors we can write the Upsilon terms in the form
\begin{equation}\label{Cexpanded}
\Upsilon'_{\frac{\gamma}{2}}(0) \Gamma_{\frac{\gamma}{2}}(\frac{\bar{\alpha}}{2}-Q) \Gamma_{\frac{\gamma}{2}}(2Q-\frac{\bar{\alpha}}{2})
\prod\limits_{i=1}^3 
\frac{   \Gamma_{\frac{\gamma}{2}}(\frac{\bar{\alpha}}{2}-\alpha_i) }{\Gamma_{\frac{\gamma}{2}}(Q-\alpha_i) }
\frac{  \Gamma_{\frac{\gamma}{2}}(Q+\alpha_i-\frac{\bar{\alpha}}{2})  }{ \Gamma_{\frac{\gamma}{2}}(\alpha_i)    }.
\end{equation}
For the ratios of the double gamma factors we have the pattern
\begin{gather}
\frac{\bar{\alpha}}{2}-\alpha_i - (Q-\alpha_i)  = \frac{\bar{\alpha}}{2} - Q, \\
Q+\alpha_i-\frac{\bar{\alpha}}{2} - \alpha_i = Q-\frac{\bar{\alpha}}{2}.
%%\sum_{i=1}^3 \frac{\bar{\alpha}}{2}-\alpha_i + Q+\alpha_i-\frac{\bar{\alpha}}{2} = \sum_{i=1}^3 Q-\alpha_i + \alpha_i. 
\end{gather}
What this means is that the difference between the arguments of the factors in the numerator and denominator is $\pm 
\bigl(\frac{\bar{\alpha}}{2} - Q\bigr),$ which is, up to sign, $\frac{\gamma}{2} s_0$ and for each ratio with
$+\bigl(\frac{\bar{\alpha}}{2} - Q\bigr)$ there is a ratio with $-\bigl(\frac{\bar{\alpha}}{2} - Q\bigr).$
Further, the same pattern is seen in the remaining factors if we write them in the form\footnote{\label{myfoot}We replaced 
$\Upsilon'_{\frac{\gamma}{2}}(0)$ with $1/\Gamma^2_{\frac{\gamma}{2}}(Q),$ cf. Lemma \ref{upsprime} below and
modified $\frac{\bar{\alpha}}{2} - Q$ to $\frac{\bar{\alpha}}{2}.$ 
By the functional equation of the double gamma function, $\Gamma(s) \Gamma(1+\frac{s}{\tau}) \Gamma_{\frac{\gamma}{2}}(\frac{\gamma}{2}s+Q) \propto \Gamma_{\frac{\gamma}{2}}(\frac{\gamma}{2}s),$ cf. \eqref{diction} and \eqref{shift}.
}
\begin{equation}
\frac{\Gamma_{\frac{\gamma}{2}}\bigl(\frac{\bar{\alpha}}{2}\bigr)}{\Gamma_{\frac{\gamma}{2}}(Q)}
\frac{\Gamma_{\frac{\gamma}{2}}(2Q-\frac{\bar{\alpha}}{2})}{\Gamma_{\frac{\gamma}{2}}(Q)}.
\end{equation}
This pattern, which is precisely the pattern of Barnes beta factors, cf. \eqref{barnesbeta22} and \eqref{barnesbeta21}, motivates the following ``analytic continuation'' from $s_0$ to complex $s,$
\begin{equation}%%\label{secondM}
\Gamma(1+\frac{s}{\tau}) 
\frac{\Gamma_{\frac{\gamma}{2}}(\frac{\gamma}{2}s+Q)}{\Gamma_{\frac{\gamma}{2}}(Q)}
\frac{\Gamma_{\frac{\gamma}{2}}(2Q-\frac{\bar{\alpha}}{2})}{\Gamma_{\frac{\gamma}{2}}(\frac{\gamma}{2}s+2Q-\frac{\bar{\alpha}}{2})}
\prod\limits_{i=1}^3 
\frac{\Gamma_{\frac{\gamma}{2}}(\frac{\gamma}{2}s+Q-\alpha_i)}{\Gamma_{\frac{\gamma}{2}}(Q-\alpha_i)}
\frac{\Gamma_{\frac{\gamma}{2}}(Q+\alpha_i-\frac{\bar{\alpha}}{2})}{\Gamma_{\frac{\gamma}{2}}(\frac{\gamma}{2}s+Q+\alpha_i-\frac{\bar{\alpha}}{2})}.
\end{equation}
This expression has the properties that when multiplied by $\Gamma(s)$ it evaluates at $s=s_0,$ up to a trivial factor, to the expression in \eqref{Cexpanded} and its value at $s=0$ is 1. Hence it provides a suitable candidate for the Mellin transform of a probability distribution
that satisfies the DOZZ formula. In fact, as the following result shows, it gives the \emph{minimal} solution.

\begin{theorem}\label{Minimal}
Define the function 
\begin{align}
\mathfrak{M}(s\,|\,\alpha, \gamma) = & \Gamma(1+\frac{s}{\tau}) 
\frac{\Gamma_2(s+1+\tau\,|\,\tau)}{\Gamma_2(1+\tau\,|\,\tau)}
\frac{\Gamma_2(2(1+\tau)-\frac{\bar{\alpha}}{\gamma}\,|\,\tau)}{\Gamma_2(s+2(1+\tau)-\frac{\bar{\alpha}}{\gamma}\,|\,\tau)} \nonumber \\
& \times \prod\limits_{i=1}^3 
\frac{\Gamma_2(s+1+\tau-\frac{2}{\gamma}\alpha_i\,|\,\tau)}{\Gamma_2(1+\tau-\frac{2}{\gamma}\alpha_i\,|\,\tau)}
\frac{\Gamma_2(1+\tau+\frac{2}{\gamma}(\alpha_i-\frac{\bar{\alpha}}{2})\,|\,\tau)}{\Gamma_2(s+1+\tau+\frac{2}{\gamma}(\alpha_i-\frac{\bar{\alpha}}{2})\,|\,\tau)}.\label{firstM}
\end{align}
Then,
\begin{enumerate}
\item The function $\mathfrak{M}(s\,|\,\alpha, \gamma)$ is analytic in $s$ over the domain specified in (\ref{mc}), (\ref{mc2})
and is the Mellin transform of a positive, log-infinitely divisible probability distribution defined as the product of independent Frechet, $\beta_{22},$ and two $\beta_{21}$ distributions. 
\begin{align}
\mathfrak{M}(s\,|\,\alpha, \gamma) = & \Gamma(1+\frac{s}{\tau}) \,
{\bf E}\Bigl[\beta_{22}^{s}\Bigl(b_0=1+\tau, \, b_1=\frac{2}{\gamma}(\alpha_1-\frac{\bar{\alpha}}{2}), \, b_2=
\frac{2}{\gamma}(\alpha_3-\frac{\bar{\alpha}}{2})\Bigr)\Bigr] \nonumber \\ \times &
{\bf E}\Bigl[\beta_{21}^{s}\Bigl(b_0=1+\tau-\frac{2}{\gamma}\alpha_1, \, b_1=1+\tau+\frac{2}{\gamma}(\alpha_1-\frac{\bar{\alpha}}{2})\Bigr)\Bigr]  \nonumber \\ \times &
{\bf E}\Bigl[\beta_{21}^{s}\Bigl(b_0=1+\tau-\frac{2}{\gamma}\alpha_3, \, b_1=\frac{2}{\gamma}(\frac{\bar{\alpha}}{2} - \alpha_1)\Bigr)\Bigr].
\end{align}
Let us divide this distribution by the constant $\frac{\pi \tau^{\frac{1}{\tau}} \Gamma\big(\frac{1}{\tau}\bigr)}{\Gamma\big(1-\frac{1}{\tau}\bigr)}$ and define the random variable $M(\alpha, \gamma)$ by
\begin{equation}\label{myfactor}
{\bf E} [M^s(\alpha, \gamma)] = \Bigl( 
\frac{\pi \tau^{\frac{1}{\tau}} \Gamma\big(\frac{1}{\tau}\bigr)}{\Gamma\big(1-\frac{1}{\tau}\bigr)}
\Bigr)^{-s} \mathfrak{M}(s\,|\,\alpha, \gamma).
\end{equation}
%%\item The function $\mathfrak{M}(s\,|\,\alpha, \gamma)$ 
\item At $s=s_0$ the Mellin transform of $M(\alpha, \gamma)$  %%the function 
%%\begin{equation}\label{myfactor}
%%%% e^{s s_0 \frac{\gamma^2}{2}\chi_g} 
%%\Bigl( 
%%\frac{\pi \tau^{\frac{1}{\tau}} \Gamma\big(\frac{1}{\tau}\bigr)}{\Gamma\big(1-\frac{1}{\tau}\bigr)}
%%\Bigr)^{-s} \mathfrak{M}(s\,|\,\alpha, \gamma)
%%%\end{equation}
takes on the value  
\begin{equation}\label{mDOZZ}
\Gamma(s) \, 
% e^{s s_0 \frac{\gamma^2}{2}\chi_g} 
{\bf E} [M^s(\alpha, \gamma)]\Big\vert_{s=s_0}  =
%%\Bigl( 
%%\frac{\pi \tau^{\frac{1}{\tau}} \Gamma\big(\frac{1}{\tau}\bigr)}{\Gamma\big(1-\frac{1}{\tau}\bigr)}
%%\Bigr)^{-s} \mathfrak{M}(s\,|\,\alpha, \gamma)\Big\vert_{s=s_0} = 
%%%%e^{\frac{s_0^2\gamma^2}{2}\chi_g}\, \, 
C_\gamma(\alpha_1, \alpha_2, \alpha_3),
\end{equation}
and so satisfies the DOZZ formula in (\ref{nogDOZZ}). It also symmetric in all $(\alpha_1, \alpha_2, \alpha_3).$
%% satisfies  %%the scaling invariance in (\ref{rhogscaling}) and  the symmetry in  \eqref{Msymmetry}. %%\eqref{a1a3}.
\item In the special case of $\alpha_i=0$  $\forall i$ (the total mass), this Mellin transform simplifies to
\begin{align}
%% e^{s s_0 \frac{\gamma^2}{2}\chi_g}
 \Bigl( 
\frac{\pi  \Gamma\big(\frac{1}{\tau}\bigr)}{\Gamma\big(1-\frac{1}{\tau}\bigr)}
\Bigr)^{-s} \, \Gamma(1+\frac{s}{\tau}) 
\frac{\Gamma(s+1+\tau)}{\Gamma(1+\tau)} \frac{\Gamma(1+\frac{s+1+\tau}{\tau})}{\Gamma(1+\frac{1+\tau}{\tau})}. \label{totalmass}
\end{align}
and corresponds to the product of three independent Frechet distributions.
\end{enumerate}
\end{theorem}

The solution in Theorem \ref{Minimal} has the limitation of being symmetric in all the $\alpha$s. The unknown
laws of  $\rho_g(\alpha_1, \alpha_2, \alpha_3, \gamma)$ corresponding to the metrics in \eqref{g+} and 
\eqref{ground} are believed to be only symmetric in $(\alpha_1, \,\alpha_3)$ for example, and 
$\rho_g(\alpha_1, \alpha_2, \alpha_3, \gamma)$ is not expected to have any particular symmetry in the $\alpha$s for a general metric. For these reasons one is interested in finding more flexible solutions. To this end we propose
a three parameter deformation of the minimal solution in Theorem \ref{Minimal}. 

Let $s_0$ be as in \eqref{s0}. Let us define the auxiliary functions depending on the deformation parameter $\rho$,
\begin{align}
\mathfrak{M}_1(s\,|\,\alpha, \gamma, \rho) = &
\frac{\Gamma_2(s+1+\tau+\rho(s-s_0)\,|\,\tau)}{\Gamma_2(1+\tau-\rho s_0\,|\,\tau)}
\frac{\Gamma_2(1+\tau+\frac{2}{\gamma}(\alpha_1-\frac{\bar{\alpha}}{2})-\rho s_0\,|\,\tau)}{\Gamma_2(s+1+\tau+\frac{2}{\gamma}(\alpha_1-\frac{\bar{\alpha}}{2})+\rho (s-s_0)\,|\,\tau)}\nonumber \\
& \times 
\frac{\Gamma_2(1+\tau+\frac{2}{\gamma}(\alpha_3-\frac{\bar{\alpha}}{2})-\rho s_0\,|\,\tau)}{\Gamma_2(s+1+\tau+\frac{2}{\gamma}(\alpha_3-\frac{\bar{\alpha}}{2})+\rho (s-s_0)\,|\,\tau)}
\frac{\Gamma_2(s+1+\tau-\frac{2}{\gamma}\alpha_2+\rho (s-s_0)\,|\,\tau)}{\Gamma_2(1+\tau-\frac{2}{\gamma}\alpha_2-\rho s_0\,|\,\tau)}
\nonumber \\
& \times 
\frac{\Gamma_2(-\rho  s+1+\tau-\frac{2}{\gamma}\alpha_2\,|\,\tau)}{\Gamma_2(1+\tau-\frac{2}{\gamma}\alpha_2\,|\,\tau)}
\frac{\Gamma_2(1+\tau+\frac{2}{\gamma}(\alpha_1-\frac{\bar{\alpha}}{2})\,|\,\tau)}{\Gamma_2(-\rho  s+1+\tau+\frac{2}{\gamma}(\alpha_1-\frac{\bar{\alpha}}{2})\,|\,\tau)} \nonumber \\
& \times 
\frac{\Gamma_2(1+\tau+\frac{2}{\gamma}(\alpha_3-\frac{\bar{\alpha}}{2})\,|\,\tau)}{\Gamma_2(-\rho  s+1+\tau+\frac{2}{\gamma}(\alpha_3-\frac{\bar{\alpha}}{2})\,|\,\tau)} 
\frac{\Gamma_2(-\rho  s+1+\tau\,|\,\tau)}{\Gamma_2(1+\tau\,|\,\tau)}, \label{aM1}
\end{align}
\begin{align}
\mathfrak{M}_{21}(s\,|\,\alpha, \gamma, \rho) =  &
\frac{\Gamma_2(s+1+\tau-\frac{2}{\gamma}\alpha_1+\rho (s-s_0)\,|\,\tau)}{\Gamma_2(1+\tau-\frac{2}{\gamma}\alpha_1-\rho s_0\,|\,\tau)}
\frac{\Gamma_2(2(1+\tau)-\frac{\bar{\alpha}}{\gamma}-\rho s_0\,|\,\tau)}{\Gamma_2(s+2(1+\tau)-\frac{\bar{\alpha}}{\gamma}+\rho (s-s_0)\,|\,\tau)}
\nonumber \\
& \times 
\frac{\Gamma_2(-\rho  s+1+\tau-\frac{2}{\gamma}\alpha_1\,|\,\tau)}{\Gamma_2(1+\tau-\frac{2}{\gamma}\alpha_1\,|\,\tau)}
\frac{\Gamma_2(2(1+\tau)-\frac{\bar{\alpha}}{\gamma}\,|\,\tau)}{\Gamma_2(-\rho  s+2(1+\tau)-\frac{\bar{\alpha}}{\gamma}\,|\,\tau)}
,\label{aM21}
\end{align}
\begin{align}
\mathfrak{M}_{22}(s\,|\,\alpha, \gamma, \rho) =  &
\frac{\Gamma_2(1+\tau+\frac{2}{\gamma}(\alpha_2-\frac{\bar{\alpha}}{2})-\rho s_0\,|\,\tau)}{\Gamma_2(s+1+\tau+\frac{2}{\gamma}(\alpha_2-\frac{\bar{\alpha}}{2})+\rho (s-s_0)\,|\,\tau)}  
\frac{\Gamma_2(s+1+\tau-\frac{2}{\gamma}\alpha_3+\rho (s-s_0)\,|\,\tau)}{\Gamma_2(1+\tau-\frac{2}{\gamma}\alpha_3-\rho s_0\,|\,\tau)}
\nonumber \\
& \times 
\frac{\Gamma_2(-\rho  s+1+\tau-\frac{2}{\gamma}\alpha_3\,|\,\tau)}{\Gamma_2(1+\tau-\frac{2}{\gamma}\alpha_3\,|\,\tau)}
\frac{\Gamma_2(1+\tau+\frac{2}{\gamma}(\alpha_2-\frac{\bar{\alpha}}{2})\,|\,\tau)}{\Gamma_2(-\rho  s+1+\tau+\frac{2}{\gamma}(\alpha_2-\frac{\bar{\alpha}}{2})\,|\,\tau)}  
. \label{aM22}
\end{align}
We now define the three parameter deformation of $\mathfrak{M}(s\,|\,\alpha, \gamma)$ by
\begin{equation}
\mathfrak{M}(s\,|\,\alpha, \gamma, \rho) = \Gamma(1+\frac{s}{\tau})  \,\mathfrak{M}_{1}(s\,|\,\alpha, \gamma, \rho_1)\,\mathfrak{M}_{21}(s\,|\,\alpha, \gamma, \rho_{21})\,\mathfrak{M}_{22}(s\,|\,\alpha, \gamma, \rho_{22}).
\end{equation}
\begin{theorem}\label{Deformation}
The function $\mathfrak{M}(s\,|\,\alpha, \gamma, \rho) $ has the following properties.
\begin{enumerate}
\item Let 
\begin{equation}
-1 < \rho_1,\,\rho_{21}, \,\rho_{22}\leq 0.
\end{equation}
The function $\mathfrak{M}(s\,|\,\alpha, \gamma, \rho) $ is analytic in $s$ over the domain specified in (\ref{mc}), (\ref{mc2}) and is the Mellin transform of a positive, log-infinitely divisible probability distribution defined as the product of independent Frechet and powers of $\beta_{22}$ and $\beta_{21}$ distributions. 
\begin{align}
\mathfrak{M}(s\,|\,\alpha, \gamma, \rho) = & \Gamma(1+\frac{s}{\tau})  \,
 {\bf E}\Bigl[\beta_{22}^{s(1+\rho_1)}\Bigl(b_0=1+\tau-\rho_1s_0, \, b_1=\frac{2}{\gamma}(\alpha_1-\frac{\bar{\alpha}}{2}), \, b_2=\frac{2}{\gamma}(\alpha_3-\frac{\bar{\alpha}}{2})
\Bigr)\Bigr]   \nonumber \\
& \times 
{\bf E}\Bigl[\beta_{22}^{-\rho_1 s}\Bigl(b_0=1+\tau-\frac{2}{\gamma}\alpha_2, \, b_1=\frac{2}{\gamma}(\frac{\bar{\alpha}}{2}-\alpha_1), \, b_2=\frac{2}{\gamma}(\frac{\bar{\alpha}}{2}-\alpha_3\Bigr)\Bigr] 
\nonumber \\
& \times 
{\bf E}\Bigl[\beta_{21}^{s(1+\rho_{21})}\Bigl(b_0=1+\tau-\frac{2}{\gamma}\alpha_1-\rho_{21}s_0, \, b_1=1+\tau+\frac{2}{\gamma}(\alpha_1-\frac{\bar{\alpha}}{2})\Bigr)\Bigr] \nonumber \\
& \times 
{\bf E}\Bigl[\beta_{21}^{-\rho_{21}s }\Bigl(b_0=1+\tau-\frac{2}{\gamma}\alpha_1, \, b_1=1+\tau+\frac{2}{\gamma}(\alpha_1-\frac{\bar{\alpha}}{2})\Bigr)\Bigr]
\nonumber \\
& \times 
{\bf E}\Bigl[\beta_{21}^{s(1+\rho_{22})}\Bigl(b_0=1+\tau-\frac{2}{\gamma}\alpha_3-\rho_{22}s_0, \, b_1=\frac{2}{\gamma}(\frac{\bar{\alpha}}{2} - \alpha_1)\Bigr)\Bigr] \nonumber \\ & \times 
{\bf E}\Bigl[\beta_{21}^{-\rho_{22} s}\Bigl(b_0=1+\tau-\frac{2}{\gamma}\alpha_3, \, b_1=\frac{2}{\gamma}(\frac{\bar{\alpha}}{2} - \alpha_1)\Bigr)\Bigr].
\end{align}
It coincides
with $\mathfrak{M}(s\,|\,\alpha, \gamma)$ in \eqref{firstM} when all the deformation parameters are zero,
\begin{equation}
\mathfrak{M}(s\,|\,\alpha, \gamma, \rho_1=0, \rho_{21}=0, \rho_{22}=0)  = 
\mathfrak{M}(s\,|\,\alpha, \gamma).
\end{equation}
and has the same value at $s=s_0,$
\begin{equation}
\mathfrak{M}(s_0\,|\,\alpha, \gamma, \rho_1, \rho_{21}, \rho_{22})  = 
\mathfrak{M}(s_0\,|\,\alpha, \gamma),
\end{equation}
so that it satisfies the DOZZ formula and the scaling invariance\footnote{
The scaling invariance is satisfied so long as $\rho_1,$ $\rho_{21},$ and $\rho_{22}$ are independent of the scale $\lambda.$ For example, if they are determined by matching some moments of $\rho_g,$ \emph{i.e.} equating \eqref{exactfirstmoment} and the corresponding moment given by our conjecture in \eqref{naivecon},
then they depend on the metric but do not change with rescaling, $\rho_i(g)=\rho_i(\lambda g)$ for any $\lambda>0$ as required by \eqref{rhogscaling}. 
%%cf. the Remark at the end of this section.
}
 in (\ref{rhogscaling})
when multiplied by  $e^{s s_0 \frac{\gamma^2}{2}\chi_g} \Bigl( 
\frac{\pi \tau^{\frac{1}{\tau}} \Gamma\big(\frac{1}{\tau}\bigr)}{\Gamma\big(1-\frac{1}{\tau}\bigr)}
\Bigr)^{-s}.$
\item The function $\mathfrak{M}(s\,|\,\alpha, \gamma, \rho_1, \rho_{21}, \rho_{22}) $ 
has the following symmetry properties. If $\rho_{21}=\rho_{22},$
\begin{equation}
\mathfrak{M}(s\,|\,\alpha_1, \alpha_2, \alpha_3, \gamma, \rho_1, \rho_{21}=\rho_{22}) = \mathfrak{M}(s\,|\,\alpha_3, \alpha_2, \alpha_1, \gamma, \rho_1, \rho_{21}=\rho_{22}).
\end{equation}
If $\rho_1=\rho_{21}=\rho_{22},$
$\mathfrak{M}(s\,|\,\alpha, \gamma, \rho_1, \rho_{21}, \rho_{22}) $ is symmetric in all the $\alpha$s.
\item In the special case of $\alpha_i=0$ $\forall i$ (the total mass), we obtain the product of five independent Frechet distributions,
\begin{align}
\mathfrak{M}(s\,|\,\alpha=0, \gamma, \rho) = & \tau^{\frac{s}{\tau}}\, \Gamma(1+\frac{s}{\tau}) \frac{\Gamma\Bigl((1+\rho_{21})(s+1+\tau)\Bigr)}{\Gamma\Bigl((1+\rho_{21})(1+\tau)\Bigr)}\frac{\Gamma\Bigl(1+\frac{(1+\rho_{21})(s+1+\tau)}{\tau}\Bigr)}{\Gamma\Bigl(1+\frac{(1+\rho_{21})(1+\tau)}{\tau}\Bigr)}
 \nonumber \\
& \times 
 \frac{\Gamma\bigl(-\rho_{21}s+1+\tau\bigr)}{\Gamma\bigl(1+\tau\bigr)}
\frac{\Gamma\Bigl(1+\frac{(-\rho_{21}s+1+\tau)}{\tau}\Bigr)}{\Gamma\Bigl(1+\frac{1+\tau}{\tau}\Bigr)}
. \label{atotalmass}
\end{align}
\item The leading asymptotic term of the Mellin transform is
\begin{equation}\label{mellinasymptotic}
\mathfrak{M}(s\,|\, \alpha,\gamma, \rho) \thicksim \exp\Bigl(\bigl(1+\frac{2}{\tau}\bigr) s\log s +O(s)\Bigr), \; \Re(s)\rightarrow +\infty.
\end{equation}
\end{enumerate}
\end{theorem}

Thus, we have constructed a family of probability distributions that satisfy the DOZZ formula and have the desired symmetry and scaling properties. 

We end this section with the special case of $s_0=0$ $(\bar{\alpha}=2Q).$ In this case we provide a modification of the minimal construction that has a different leading asymptotic of the Mellin transform. 
\begin{theorem}\label{Minimals0}
Let $s_0=0.$ Let $\mathfrak{M}(s\,|\,\alpha, \gamma) $ denote the minimal solution in \eqref{firstM} and define the function 
\begin{equation}
\mathfrak{M}'(s\,|\,\alpha, \gamma) = 
\frac{\Gamma_2(s+2(1+\tau)-\frac{2}{\gamma}\alpha_3\,|\,\tau)}{\Gamma_2(2(1+\tau)-\frac{2}{\gamma}\alpha_3\,|\,\tau)}
\frac{\Gamma_2(1+\tau-\frac{2}{\gamma}\alpha_3\,|\,\tau)}{\Gamma_2(s+1+\tau-\frac{2}{\gamma}\alpha_3\,|\,\tau)}
\mathfrak{M}(s\,|\,\alpha, \gamma).\label{firstMs0}
\end{equation}
Then, $\mathfrak{M}'(s\,|\,\alpha, \gamma)$ is analytic in $s$ over the domain specified in (\ref{mc}), (\ref{mc2}), 
satisfies the DOZZ formula in the sense of \eqref{myfactor} and \eqref{mDOZZ}, and is the Mellin transform of a positive, log-infinitely divisible probability distribution defined as the product of independent Frechet and two  $\beta_{22}$ distributions. 
\begin{align}
\mathfrak{M}'(s\,|\,\alpha, \gamma) = & \Gamma(1+\frac{s}{\tau}) \,
{\bf E}\Bigl[\beta_{22}^{s}\Bigl(b_0=1+\tau, \, b_1=\frac{2}{\gamma}(\alpha_1-\frac{\bar{\alpha}}{2}), \, b_2=
\frac{2}{\gamma}(\alpha_3-\frac{\bar{\alpha}}{2})\Bigr)\Bigr] \nonumber \\ \times &
{\bf E}\Bigl[\beta_{22}^{s}\Bigl(b_0=1+\tau-\frac{2}{\gamma}\alpha_1, \, b_1=1+\tau+\frac{2}{\gamma}(\alpha_1-\frac{\bar{\alpha}}{2}), 
b_2=  \frac{2}{\gamma}(\frac{\bar{\alpha}}{2} -  \alpha_3) \Bigr)\Bigr].
\end{align}
The leading asymptotic term of the Mellin transform is
\begin{equation}
\mathfrak{M}'(s\,|\, \alpha,\gamma, \rho) \thicksim \exp\Bigl(\frac{s}{\tau} \log s +O(s)\Bigr), \; \Re(s)\rightarrow +\infty.
\end{equation}
\end{theorem}
The modified minimal construction is symmetric in $(\alpha_1,\alpha_2).$ It is obvious that by changing the index of $\alpha_i$ in the pre-factor in front of 
 $\mathfrak{M}(s\,|\,\alpha, \gamma) $ in \eqref{firstMs0} one obtains solutions that are symmetric in any given pair of $\alpha_i$s. %One can also apply the type of transformations that were used in Theorem \ref{Deformation} to produce other solutions with different symmetry types. 

\section{Proofs}

We begin by analyzing the natural boundaries in \eqref{mc} and \eqref{mc2} and the implications of $s_0$ satisfying these conditions in detail. 
\begin{lemma}\label{Inequalitiies}
Let $s$ and $s_0$ satisfy \eqref{mc} and \eqref{mc2}, $\alpha_i$ satisfy \eqref{bin1} $\forall i.$  Then, 
the following inequalities hold for the $\alpha$s:
\begin{gather}
1+\tau-\frac{2\alpha_k}{\gamma}>0, \label{bin3} \\
\bar{\alpha}>2\alpha_k, \label{bin2} \\
%%\frac{\bar{\alpha}}{\gamma}>1.
2Q+\alpha_k-\bar{\alpha}>0, \label{in1}\\
\alpha_k>0,  \label{in2} \\
Q+\alpha_k-\frac{\bar{\alpha}}{2}>0,  \label{in3} \\
1+\tau+\frac{2}{\gamma}(\alpha_k-\frac{\bar{\alpha}}{2})>0, \label{eqcondit3} \\
1<  \frac{\bar{\alpha}}{\gamma}<\frac{3}{2}(1+\tau), \label{alphagamma2}
\end{gather}
and the following inequalities\footnote{To ease notation we write $s$ instead of $\Re(s)$ here and throughout this section.} hold for $s:$
\begin{gather}
%%1+\frac{s}{\tau} > 0, \\
s+ 1 + \tau - \frac{2\alpha_k}{\gamma} > 0,\label{sin} \\
s+ 1+\tau + \frac{2}{\gamma} (\alpha_k-\frac{\bar{\alpha}}{2}) >0, \label{sadd} \\
s+ 1 + \tau - \frac{2\bar{\alpha}}{3\gamma} >0, \label{sadd2} \\
s +  2(1+\tau) - \frac{\bar{\alpha}}{\gamma} > 0, \label{sadd3}
\end{gather}
and the following inequalities hold for $s_0:$
\begin{gather}
s_0 < \frac{1+\tau}{2}, \label{alphagamma} \\
s_0 < \frac{\alpha_k}{\gamma}. \label{gamma}
\end{gather}
\end{lemma}
\begin{proof}
Let $(i, j, k)$ denote an arbitrary permutation of $(1, 2, 3).$
Observe the identity
\begin{equation}
\frac{\gamma}{2}(1+\tau) = Q.\label{Q}
\end{equation}
We note first that \eqref{bin3} is equivalent to \eqref{bin1} by \eqref{Q}. \eqref{bin2} follows immediately from $s_0$ satisfying \eqref{mc2}. Now,
\begin{align}
2\alpha_k-\bar{\alpha} & = 2\alpha_k - (\alpha_i+\alpha_j+\alpha_k), \\
& = \alpha_k-(\alpha_i+\alpha_j), \\
& >\alpha_k-2Q\; %%\text{as} \; \alpha_i<Q.
\end{align}
due to \eqref{bin1}.
This verifies \eqref{in1}. Next, 
\begin{align}
2\alpha_k-\bar{\alpha} & = 2\alpha_k - (\alpha_i+\alpha_j+\alpha_k), \\
& = \alpha_k - (\alpha_i+\alpha_j), \\
& > \alpha_k - \bar{\alpha}
\end{align}
due to \eqref{bin2}.
This verifies \eqref{in2}. To verify \eqref{in3},
\begin{align}
Q-\frac{\bar{\alpha}}{2}+\alpha_k &= \frac{\bar{\alpha}}{2} + Q - (\alpha_i+\alpha_j), \\
& = (\frac{\bar{\alpha}}{2} -\alpha_i ) + (Q-\alpha_j) >0
\end{align}
due to \eqref{bin1} and \eqref{bin2}. It also follows directly from \eqref{in1} and \eqref{in2}. \eqref{eqcondit3}
is equivalent to \eqref{in3} by means of \eqref{Q}. Finally, \eqref{alphagamma2} follows from $s_0$ satisfying \eqref{mc} and
$\bar{\alpha}<3Q.$

The inequality in \eqref{sin} is equivalent to \eqref{mc2}.
To verify \eqref{sadd},%% and \eqref{sadd2}, 
\begin{align}
s+ 1+\tau + \frac{2}{\gamma} (\alpha_k-\frac{\bar{\alpha}}{2})  & = s+ 1+\tau + \frac{2}{\gamma} (\bar{\alpha}-\alpha_i-\alpha_j-\frac{\bar{\alpha}}{2}), \\
& = s+ 1+\tau - \frac{2\alpha_i}{\gamma}  + \frac{2}{\gamma} (\frac{\bar{\alpha}}{2}-\alpha_j) >0
\end{align}
by (\ref{bin2}) and (\ref{sin}). (\ref{sadd2}) follows from (\ref{sin}) and the definition of $\bar{\alpha}.$
To verify \eqref{sadd3}, 
\begin{equation}
 s - \frac{\bar{\alpha}}{\gamma} + 2(1+\tau) =  \bigl(s  + 1+\tau - \frac{2\bar{\alpha}}{3\gamma}\bigr)  + \bigl(1+\tau - \frac{\bar{\alpha}}{3\gamma}\bigr) >0
\end{equation}
by (\ref{sadd2}) and (\ref{alphagamma2}).

The inequality in \eqref{alphagamma} follows from \eqref{alphagamma2}. \eqref{gamma} follows from  (\ref{in1}). 
\qed
\end{proof}

The next auxiliary result that we need is the computation of the derivative of the Upsilon function in the DOZZ formula.
\begin{lemma}\label{upsprime}
\begin{equation}
\Upsilon_{\frac{\gamma}{2}}' (0) = \frac{2\pi}{\Gamma^2_{\frac{\gamma}{2}}(Q)}.
\end{equation}
\end{lemma}
\begin{proof}
The proof of this formula is a corollary of the functional equations of the double and Euler's gamma functions. We start with the definition of the Upsilon function
in (\ref{ups}),
\begin{equation}
\Upsilon_{\frac{\gamma}{2}} (x) = \frac{1}{\Gamma_{\frac{\gamma}{2}}(x)\,\Gamma_{\frac{\gamma}{2}}(Q-x)}.
\end{equation}
The singularity at $x=0$ comes from the first factor. It is sufficient to show that
\begin{equation}
\log \Gamma_{\frac{\gamma}{2}}(x) =  -\log x + C + O(x), \; x\rightarrow 0, \label{mainest}
\end{equation}
for some constant $C.$
Then,
\begin{equation}
\Upsilon_{\frac{\gamma}{2}}' (0) = \frac{e^{-C}}{\Gamma_{\frac{\gamma}{2}}(Q)}.
\end{equation}
We will now show that (\ref{mainest}) holds with 
\begin{equation}
C = \log \Gamma_{\frac{\gamma}{2}}(Q) - \log 2\pi.
\end{equation}
%%The functional equations of the double gamma function are
%%\begin{align}
%%\frac{\Gamma_2(s+\tau|\tau)}{\Gamma_2(s+\tau+1|\tau)} & = \frac{\tau^{\frac{s+\tau}{\tau}-\frac{1}{2}}}{\sqrt{2\pi}} \, \Gamma\bigl(\frac{s+\tau}{\tau}\bigr), \label{doublefunceq1}\\
%\frac{\Gamma_2(s|\tau)}{\Gamma_2(s+\tau|\tau)} & = \frac{1}{\sqrt{2\pi}}\,\Gamma(s).  \label{doublefunceq2}
%%\end{align}
Multiplying together the functional equations of the double gamma function, cf. \eqref{doublefunceq1} and \eqref{doublefunceq2}, we obtain
\begin{equation}\label{shift}
\frac{\Gamma_2(s|\tau)}{\Gamma_2(s+\tau+1|\tau)} =  \frac{\tau^{\frac{s}{\tau}+\frac{1}{2}}}{2\pi}\, \Gamma(s)  \Gamma\bigl(\frac{s+\tau}{\tau}\bigr).
\end{equation}
By taking the log of this equation and using the functional equation of Euler's gamma function, we get
\begin{equation}
\log\Gamma_2(s|\tau) =  \log\Gamma_2(s+\tau+1|\tau) +  \log\frac{\tau^{\frac{s}{\tau}+\frac{1}{2}}}{2\pi} +\log\Gamma(s+1) -\log(s) +\log\Gamma\bigl(\frac{s+\tau}{\tau}\bigr).
\end{equation}
Thus,
\begin{equation}
\log\Gamma_2(s|\tau) = -\log s +  \log\frac{\tau^{\frac{1}{2}}}{2\pi} + \log\Gamma_2(\tau+1|\tau)  + O(s), \; s\rightarrow 0.
\end{equation}
It now remains to translate this into the equivalent asymptotic for $\log \Gamma_{\frac{\gamma}{2}}(x).$ Recalling (\ref{diction}),
\begin{align}
\log\Gamma_{\frac{\gamma}{2}}(x) = & \log(\frac{2}{\gamma})^{\frac{1}{2} (x-\frac{Q}{2})^2} + \log\Gamma_2(\frac{2x}{\gamma}|\tau) - \log\Gamma_2(\frac{Q}{\gamma}|\tau), \nonumber \\
= & \log(\frac{2}{\gamma})^{\frac{1}{2} (x-\frac{Q}{2})^2}  - \log \frac{2x}{\gamma} +  \log\frac{\tau^{\frac{1}{2}}}{2\pi} + \log\Gamma_2(\tau+1|\tau)   - \log\Gamma_2(\frac{Q}{\gamma}|\tau) +O(x).
\end{align}
On the other hand,
\begin{equation}
\log\Gamma_{\frac{\gamma}{2}}(Q) =  \log(\frac{2}{\gamma})^{\frac{1}{2} (Q-\frac{Q}{2})^2} + \log\Gamma_2(1+\tau|\tau) - \log\Gamma_2(\frac{Q}{\gamma}|\tau).
\end{equation}
The result now follows from recalling the definition of $\tau$ in (\ref{taudef}). \qed
\end{proof}

We can now give the proof of Theorem \ref{Minimal}.
\begin{proof}

%%The proof of Part 1 is based on \ref{barnesbeta22}, \ref{barnesbeta21}, and \ref{gammaid}. 
We note first that the real part of the arguments of the double gamma factors that enter \eqref{firstM} 
is positive for each double gamma factor due to \eqref{mc}, \eqref{mc2}, \eqref{bin1} or their corollaries in
Lemma \ref{Inequalitiies}. Hence, $\mathfrak{M}(s\,|\,\alpha, \gamma) $ is analytic over the domain specified in
\eqref{mc}, \eqref{mc2}. 

Next, we split the expression in \eqref{firstM} into two groups
of factors, which we denote by $\mathfrak{M}_1(s\,|\,\alpha, \gamma)$ and $\mathfrak{M}_2(s\,|\,\alpha, \gamma),$
\begin{align}
\mathfrak{M}_1(s\,|\,\alpha, \gamma) = &
\frac{\Gamma_2(s+1+\tau\,|\,\tau)}{\Gamma_2(1+\tau\,|\,\tau)}
\frac{\Gamma_2(1+\tau+\frac{2}{\gamma}(\alpha_1-\frac{\bar{\alpha}}{2})\,|\,\tau)}{\Gamma_2(s+1+\tau+\frac{2}{\gamma}(\alpha_1-\frac{\bar{\alpha}}{2})\,|\,\tau)}
\frac{\Gamma_2(1+\tau+\frac{2}{\gamma}(\alpha_3-\frac{\bar{\alpha}}{2})\,|\,\tau)}{\Gamma_2(s+1+\tau+\frac{2}{\gamma}(\alpha_3-\frac{\bar{\alpha}}{2})\,|\,\tau)}\nonumber \\
& \times 
\frac{\Gamma_2(s+1+\tau-\frac{2}{\gamma}\alpha_2\,|\,\tau)}{\Gamma_2(1+\tau-\frac{2}{\gamma}\alpha_2\,|\,\tau)}, \label{M1}
\end{align}
\begin{align}
\mathfrak{M}_2(s\,|\,\alpha, \gamma) =  &
\frac{\Gamma_2(s+1+\tau-\frac{2}{\gamma}\alpha_1\,|\,\tau)}{\Gamma_2(1+\tau-\frac{2}{\gamma}\alpha_1\,|\,\tau)}
\frac{\Gamma_2(2(1+\tau)-\frac{\bar{\alpha}}{\gamma}\,|\,\tau)}{\Gamma_2(s+2(1+\tau)-\frac{\bar{\alpha}}{\gamma}\,|\,\tau)}
\frac{\Gamma_2(1+\tau+\frac{2}{\gamma}(\alpha_2-\frac{\bar{\alpha}}{2})\,|\,\tau)}{\Gamma_2(s+1+\tau+\frac{2}{\gamma}(\alpha_2-\frac{\bar{\alpha}}{2})\,|\,\tau)}\nonumber \\
& \times
\frac{\Gamma_2(s+1+\tau-\frac{2}{\gamma}\alpha_3\,|\,\tau)}{\Gamma_2(1+\tau-\frac{2}{\gamma}\alpha_3\,|\,\tau)},
\label{M2}
\end{align}
so that
\begin{equation}\label{MM1M2}
\mathfrak{M}(s\,|\,\alpha, \gamma) = \Gamma(1+\frac{s}{\tau})  \,\mathfrak{M}_1(s\,|\,\alpha, \gamma) \,\mathfrak{M}_2(s\,|\,\alpha, \gamma).
\end{equation}
It is sufficient to show that $\mathfrak{M}_1(s\,|\,\alpha, \gamma)$ and $\mathfrak{M}_2(s\,|\,\alpha, \gamma)$ are Mellin transforms of positive, log-infinitely
divisible distributions. For $\mathfrak{M}_1(s\,|\,\alpha, \gamma)$ this follows directly from \eqref{barnesbeta22}. Denote
\begin{align}
b_0 & = 1+\tau, \\
b_1 & = \frac{2}{\gamma}(\alpha_1-\frac{\bar{\alpha}}{2}), \\
b_2 & = \frac{2}{\gamma}(\alpha_3-\frac{\bar{\alpha}}{2}).
\end{align}
Then,  $b_0>0,$ $b_0+b_1>0,$  $b_0+b_2>0,$  $b_0+b_1+b_2>0$ and $b_1b_2>0$ due to Lemma \ref{Inequalitiies} and the expression in \eqref{barnesbeta22} with these values of the $b$s coincides
with $\mathfrak{M}_1(s\,|\,\alpha, \gamma),$ hence
\begin{equation}
\mathfrak{M}_{1}(s\,|\,\alpha, \gamma) =  {\bf E}\Bigl[\beta_{22}^{s}\Bigl(b_0=1+\tau, \, b_1=\frac{2}{\gamma}(\alpha_1-\frac{\bar{\alpha}}{2}), \, b_2=
\frac{2}{\gamma}(\alpha_3-\frac{\bar{\alpha}}{2})
\Bigr)\Bigr].
\end{equation}

The argument for $\mathfrak{M}_2(s\,|\,\alpha, \gamma)$ is similar but requires an extra step. 
To prove that $\mathfrak{M}_2(s\,|\,\alpha, \gamma)$ is the Mellin transform we split it further into two groups of factors,
\begin{align}
\mathfrak{M}_{21}(s\,|\,\alpha, \gamma) =  &
\frac{\Gamma_2(s+1+\tau-\frac{2}{\gamma}\alpha_1\,|\,\tau)}{\Gamma_2(1+\tau-\frac{2}{\gamma}\alpha_1\,|\,\tau)}
\frac{\Gamma_2(2(1+\tau)-\frac{\bar{\alpha}}{\gamma}\,|\,\tau)}{\Gamma_2(s+2(1+\tau)-\frac{\bar{\alpha}}{\gamma}\,|\,\tau)},\label{M21} \\
\mathfrak{M}_{22}(s\,|\,\alpha, \gamma) =  &
\frac{\Gamma_2(1+\tau+\frac{2}{\gamma}(\alpha_2-\frac{\bar{\alpha}}{2})\,|\,\tau)}{\Gamma_2(s+1+\tau+\frac{2}{\gamma}(\alpha_2-\frac{\bar{\alpha}}{2})\,|\,\tau)}  \frac{\Gamma_2(s+1+\tau-\frac{2}{\gamma}\alpha_3\,|\,\tau)}{\Gamma_2(1+\tau-\frac{2}{\gamma}\alpha_3\,|\,\tau)}, \label{M22}
\end{align}
so that
\begin{equation}\label{MM1M21M22}
\mathfrak{M}(s\,|\,\alpha, \gamma) = \Gamma(1+\frac{s}{\tau})  \,\mathfrak{M}_1(s\,|\,\alpha, \gamma) \,\mathfrak{M}_{21}(s\,|\,\alpha, \gamma) \, \mathfrak{M}_{22}(s\,|\,\alpha, \gamma).
\end{equation}
Let
\begin{align}
b_0 & = 1+\tau-\frac{2}{\gamma}\alpha_1, \\
b_1 & = 1+\tau+\frac{2}{\gamma}(\alpha_1-\frac{\bar{\alpha}}{2}).
\end{align}
Then, (\ref{M21}) is a special case of (\ref{barnesbeta21}) with these values of $b_0$ and $b_1.$ They are positive by Lemma \ref{Inequalitiies}, hence
\begin{equation}
\mathfrak{M}_{21}(s\,|\,\alpha, \gamma) =  {\bf E}\Bigl[\beta_{21}^{s}\Bigl(b_0=1+\tau-\frac{2}{\gamma}\alpha_1, \, b_1=1+\tau+\frac{2}{\gamma}(\alpha_1-\frac{\bar{\alpha}}{2})\Bigr)\Bigr].
\end{equation}
Now, let
\begin{align}
b_0 & = 1+\tau-\frac{2}{\gamma}\alpha_3, \\
b_1 & = \frac{2}{\gamma}(\frac{\bar{\alpha}}{2} - \alpha_1).
\end{align}
Then, (\ref{M22}) is a special case of (\ref{barnesbeta21}) with these values of $b_0$ and $b_1.$ Indeed, these are positive  by Lemma \ref{Inequalitiies}  and
\begin{equation}
b_0+b_1 = 1+\tau+\frac{2}{\gamma}(\alpha_2-\frac{\bar{\alpha}}{2}) 
\end{equation}
by the definition of $\bar{\alpha},$ hence,
\begin{equation}
\mathfrak{M}_{22}(s\,|\,\alpha, \gamma) =  {\bf E}\Bigl[\beta_{21}^{s}\Bigl(b_0=1+\tau-\frac{2}{\gamma}\alpha_3, \, b_1=\frac{2}{\gamma}(\frac{\bar{\alpha}}{2} - \alpha_1)\Bigr)\Bigr].
\end{equation}
This completes the proof of Part I. 

To prove Part II we first need to establish the identity
%%The proof of Part II is based on Lemma \ref{upsprime}. We first need to establish the identity
\begin{align}
\Gamma(s) \, \mathfrak{M}(s\,|\,\alpha, \gamma)\Big\vert_{s=s_0} = & 
2\pi \tau^{-\frac{s_0}{\tau}-\frac{1}{2}}\,
\frac{\Gamma_2(s_0|\tau)}{\Gamma_2(1+\tau|\tau)}
\, \frac{\Gamma_2(1+\tau-s_0|\tau)}{\Gamma_2(1+\tau|\tau)}\, \nonumber \\
& \times 
\prod\limits_{i=1}^3 \frac{\Gamma_2(1+\tau+\frac{2}{\gamma}(\alpha_i-\frac{\bar{\alpha}}{2})\,|\,\tau)}{\Gamma_2(s_0+1+\tau+\frac{2}{\gamma}(\alpha_i-\frac{\bar{\alpha}}{2})\,|\,\tau)}
\,
\prod\limits_{i=1}^3 \frac{\Gamma_2(s_0+1+\tau-\frac{2}{\gamma}\alpha_i\,|\,\tau)}{\Gamma_2(1+\tau-\frac{2}{\gamma}\alpha_i\,|\,\tau)}
,
\end{align}
which follows from \eqref{shift} by means of
\begin{equation}
\Gamma(s_0)\, \Gamma(1+\frac{s_0}{\tau}) \,\frac{\Gamma_2(s_0+1+\tau\,|\,\tau)}{\Gamma_2(1+\tau\,|\,\tau)} = 2\pi \tau^{-\frac{s_0}{\tau}-\frac{1}{2}}\,
\frac{\Gamma_2(s_0|\tau)}{\Gamma_2(1+\tau|\tau)}.
\end{equation}
Applying (\ref{diction}), we have thus established
\begin{equation}
\Gamma(s) \, \mathfrak{M}(s\,|\,\alpha, \gamma)\Big\vert_{s=s_0} =
2\pi \tau^{-\frac{s_0}{\tau}-\frac{1}{2}}\, \bigl(  \frac{\gamma}{2}\bigr)^{-s_0 Q\gamma}   
\frac{\prod\limits_{i=1}^3 \Upsilon_{\frac{\gamma}{2}} (\alpha_i)}
{\Gamma^2_{\frac{\gamma}{2}}(Q) \Upsilon_{\frac{\gamma}{2}} (\frac{\bar{\alpha}}{2}-Q)
\prod\limits_{i=1}^3 \Upsilon_{\frac{\gamma}{2}} (\frac{\bar{\alpha}}{2}-\alpha_i)}.
\end{equation}
Applying Lemma \ref{upsprime} and simplifying, we obtain
\begin{equation}
\Gamma(s) \, \mathfrak{M}(s\,|\,\alpha, \gamma)\Big\vert_{s=s_0} =  \bigl(\frac{\gamma}{2}\bigr)^{-2 s_0}\, \frac{\gamma}{2}
\,\frac{\Upsilon'_{\frac{\gamma}{2}}(0)\prod\limits_{i=1}^3 \Upsilon_{\frac{\gamma}{2}} (\alpha_i)}
{ \Upsilon_{\frac{\gamma}{2}} (\frac{\bar{\alpha}}{2}-Q)
\prod\limits_{i=1}^3 \Upsilon_{\frac{\gamma}{2}} (\frac{\bar{\alpha}}{2}-\alpha_i)}.
\end{equation}
Upon comparing with the right-hand side of the DOZZ formula, we see that they differ by the factor
\begin{equation}
%% e^{ s_0^2 \frac{\gamma^2}{2}\chi_g}
 \Bigl(  \pi  l(\frac{\gamma^2}{4}) (\frac{\gamma}{2})^{-\frac{\gamma^2}{2}}\Bigr)^{-s_0}
\end{equation}
which is precisely the factor in \eqref{myfactor} due to the identity
\begin{equation}
(\frac{\gamma}{2})^{-\frac{\gamma^2}{2}} = \tau^{\frac{1}{\tau}}.
\end{equation}
The scaling invariance in \eqref{rhogscaling} follows from \eqref{lambdascaling}. 
The symmetry in \eqref{a1a3} is a corollary of \eqref{MM1M2} as both $\mathfrak{M}_1(s\,|\,\alpha, \gamma)$ and $\mathfrak{M}_2(s\,|\,\alpha, \gamma)$ are symmetric under $\alpha_1\leftrightarrow \alpha_3.$
Finally, \eqref{totalmass} follows from \eqref{shift}.
\qed
\end{proof}

We now proceed to the proof of Theorem \ref{Deformation} and start with a series of lemmas. The first lemma extends Lemma 
\ref{Inequalitiies}.
\begin{lemma}\label{Inequalitiies2}
Let $-1\leq\rho\leq 0$ and 
\begin{gather}
E>0, \\
s+E>0, \\
s_0+E>0.
\end{gather}
Then, 
\begin{gather}
-\rho s+ E>0, \label{E} \\
s+E + \rho(s-s_0) >0. \label{E2}
\end{gather}
In particular, all arguments of the double gamma factors in \eqref{aM1} -- \eqref{aM22} are positive if $s$ and $s_0$ satisfy \eqref{mc} and \eqref{mc2}, $-1\leq\rho\leq 0$  and the $\alpha$s satisfy \eqref{bin1}.
\end{lemma}
\begin{proof}If $\rho=0,$ there is nothing to show. Else, $-\rho s > \rho E$ because $\rho< 0$ 
so that $-\rho s+E >  (1+\rho)E\geq 0,$ because $\rho\geq -1$ and $E>0$ as desired. Similarly,  $-\rho s_0 > \rho E$
so that $s+E+\rho(s-s_0) > (1+\rho)(s+E)\geq 0.$ To apply these inequalities to the arguments of the double gamma factors in 
\eqref{aM1} -- \eqref{aM22} it is sufficient to note that $E$ corresponds to one of the
expressions $1+\tau,$ $1+\tau-\frac{2\alpha_k}{\gamma},$ $1+\tau+\frac{2}{\gamma}(\alpha_k-\frac{\bar{\alpha}}{2}),$ and 
$2(1+\tau) - \frac{\bar{\alpha}}{\gamma}.$ By Lemma \ref{Inequalitiies} we know that $E>0$ and $s+E>0$ for any $s$ that
satisfies \eqref{mc} and \eqref{mc2} so in particular $s_0+E>0.$ The result now follows from \eqref{E} and \eqref{E2}.
\qed
\end{proof}

The next lemmas provide deformations of the Barnes beta distributions of types $(2,1)$ and $(2,2)$ in \eqref{barnesbeta21}
and \eqref{barnesbeta22}.
\begin{lemma}\label{deformation21}
Let $b_0,b_1>0$ and $\beta_{21}(b_0, b_1)$ be a Barnes beta distribution of type $(2,1).$ Let 
\begin{gather}
-1\leq\rho\leq 0, \label{in21rho} \\
s_0+b_0>0, \label{in21s0} 
%%b_0-\rho s_0>0, \label{in21b0}
\end{gather}
%$-1<\rho\leq 0, $ $s_0+b_0>0,$ 
and consider the following deformation of the Mellin transform of  $\beta_{21}(b_0, b_1),$
%%which entails adding $\rho(s-s_0)$ to the argument of each double gamma factor,
\begin{align}
\eta_{21}(s, \rho) = &
\frac{\Gamma_2(s(1+\rho)+b_0-\rho s_0\,|\,\tau)}{\Gamma_2(b_0-\rho s_0\,|\,\tau)}
\frac{\Gamma_2( b_0+b_1-\rho s_0\,|\,\tau)}{\Gamma_2(s(1+\rho)+b_0+b_1-\rho s_0\,|\,\tau)}\nonumber \\
& \times 
\frac{\Gamma_2(-\rho s+b_0\,|\,\tau)}{\Gamma_2(b_0\,|\,\tau)}
\frac{\Gamma_2( b_0+b_1\,|\,\tau)}{\Gamma_2(-\rho s+b_0+b_1\,|\,\tau)}. 
\label{barnesbeta21Deform}
\end{align}
Then, %%for sufficiently small $\rho$ 
this deformation is analytic over the same domain as the original Mellin transform, $s+b_0>0,$ is the Mellin transform of the following random variable,
\begin{equation}
\beta_{21}^{1+\rho}(b_0-\rho s_0, b_1)\,\beta_{21}^{-\rho}(b_0, b_1),
\end{equation}
and this deformation has the same value at $s=s_0$ as the original Mellin transform,
\begin{equation}\label{s0invariant21}
\eta_{21}(s_0, \rho) = {\bf E}[\beta^{s_0}_{21}(b_0, b_1)].
\end{equation}
\end{lemma}
\begin{proof}
By Lemma \ref{Inequalitiies2} the conditions in \eqref{in21rho} and \eqref{in21s0} and $b_0, \,b_1>0$ guarantee that $\eta_{21}(s,\rho)$ is analytic over $s + b_0>0,$ and $b_0-\rho s_0>0$ so that $\beta_{21}(b_0-\rho s_0, b_1)$ is well-defined. 
%%In fact, $s_0+b_0>0$ implies $-\rho s_0 > \rho b_0$ as $\rho<0.$ Hence, $s+b_0+\rho(s-s_0) > (1+\rho)(s+b_0)>0$ by
%assumption. Similarly,  $s+b_0>0$ implies $-\rho s > \rho b_0,$ hence $-\rho s+ b_0 > (1+\rho)b_0>0.$ 
%The condition in \eqref{in21b0} guarantees that $\beta_{21}(b_0-\rho s_0, b_1)$ is well-defined. 
The equality
\begin{equation}
\eta_{21}(s, \rho) = {\bf E}\Bigl[  \Bigl(\beta_{21}^{1+\rho}(b_0-\rho s_0, b_1) \Bigr)^s\Bigr]\,{\bf E}\Bigl[\Bigl(\beta_{21}^{-\rho}(b_0, b_1)\Bigr)^s  \Bigr], \; s + b_0>0,
\end{equation}
and \eqref{s0invariant21} follow from \eqref{barnesbeta21} by direct inspection.
\qed
\end{proof}

\begin{lemma}\label{deformation22}
Let $b_0>0, b_0+b_1, b_0+b_2, b_0+b_1+b_2, b_1b_2>0$ and 
$\beta_{22}(b_0, b_1, b_2)$ be a Barnes beta distribution of type $(2,2).$ Let 
\begin{gather}
-1\leq\rho\leq 0, \label{in22rho} \\
s_0+b_0>0, \; s_0+b_0+b_1>0, \; s_0+b_0+b_2>0, \; s_0+b_0+b_1+b_2>0,    \label{in22s0} 
%%b_0-\rho s_0>0, \; b_0+b_1 - \rho s_0>0, \; b_0+b_2 - \rho s_0>0, \; b_0+b_1+b_2 - \rho s_0>0, \label{in22b0}
\end{gather}
and consider the following deformation of the Mellin transform of  $\beta_{22}(b_0, b_1, b_2),$
%%which entails adding $\rho(s-s_0)$ to the argument of each double gamma factor,
\begin{align}
\eta_{22}(s, \rho) = &
\frac{\Gamma_2(s(1+\rho)+b_0-\rho s_0\,|\,\tau)}{\Gamma_2(b_0-\rho s_0\,|\,\tau)}
\frac{\Gamma_2(b_0+b_1-\rho s_0\,|\,\tau)}{\Gamma_2(s(1+\rho)+b_0+b_1-\rho s_0\,|\,\tau)}\nonumber \\
& \times 
\frac{\Gamma_2(b_0+b_2-\rho s_0\,|\,\tau)}{\Gamma_2(s(1+\rho)+b_0+b_2-\rho s_0\,|\,\tau)}
\frac{\Gamma_2(s(1+\rho)+b_0+b_1+b_2 - \rho s_0\,|\,\tau)}{\Gamma_2(b_0+b_1+b_2-\rho s_0\,|\,\tau)}\nonumber \\
& \times 
\frac{\Gamma_2(-\rho s+b_0\,|\,\tau)}{\Gamma_2(b_0\,|\,\tau)}
\frac{\Gamma_2(b_0+b_1\,|\,\tau)}{\Gamma_2(-\rho  s+b_0+b_1\,|\,\tau)}\nonumber \\
& \times 
\frac{\Gamma_2(b_0+b_2\,|\,\tau)}{\Gamma_2(-\rho s+b_0+b_2\,|\,\tau)}
\frac{\Gamma_2(-\rho s+b_0+b_1+b_2\,|\,\tau)}{\Gamma_2(b_0+b_1+b_2\,|\,\tau)}.
\label{barnesbeta22Deform}
\end{align}
Then, this deformation is analytic over the same domain as the original Mellin transform, $s+b_0>0,$ $s+b_0+b_1>0,$ $s+b_0+b_2>0,$ $s+b_0+b_1+b_2>0,$
is the Mellin transform of the following random variable,
\begin{equation}
\beta_{22}^{1+\rho}(b_0-\rho s_0, b_1, b_2)\,\beta_{22}^{-\rho}(b_0+b_1+b_2, -b_1, -b_2),
\end{equation}
and this deformation has the same value at $s=s_0$ as the original Mellin transform,
\begin{equation}\label{s0invariant22}
\eta_{22}(s_0, \rho) = {\bf E}[\beta^{s_0}_{22}(b_0, b_1, b_2)].
\end{equation}
\end{lemma}
\begin{proof}
By Lemma \ref{Inequalitiies2} the conditions in \eqref{in22rho} and \eqref{in22s0} guarantee that $\eta_{22}(s,\rho)$ is analytic over $s+b_0>0,$ $s+b_0+b_1>0,$ $s+b_0+b_2>0,$ $s+b_0+b_1+b_2>0$ and $b_0-\rho s_0>0$,  $b_0+b_1-\rho s_0>0$,
$b_0+b_2-\rho s_0>0$,  $b_0+b_1+b_2-\rho s_0>0$, so that $\beta_{22}(b_0-\rho s_0, b_1, b_2)$
is well-defined. We also note that $\beta_{22}(b_0+b_1+b_2, -b_1, -b_2)$ is also well defined because $b_0>0,$ $b_0+b_1>0,$ 
$b_0+b_2>0,$ $b_0+b_1+b_2>0,$ and $b_1 b_2>0$ by assumption. 
%%The argument is the same as in Lemma \ref{deformation21}.
%%The conditions in \eqref{in22b0} guarantee that $\beta_{22}(b_0-\rho s_0, b_1, b_2)$ is well-defined. 
The equality
\begin{equation}
\eta_{22}(s, \rho) = {\bf E}\Bigl[  \Bigl(\beta_{22}^{1+\rho}(b_0-\rho s_0, b_1, b_2)\Bigr)^s\Bigr]\,{\bf E}\Bigl[\Bigl(\beta_{22}^{-\rho}(b_0+b_1+b_2, -b_1, -b_2)\Bigr)^s  \Bigr]
\end{equation}
over the domain $s+b_0>0,$ $s+b_0+b_1>0,$ $s+b_0+b_2>0,$ $s+b_0+b_1+b_2>0$ 
and \eqref{s0invariant22} follow from \eqref{barnesbeta22} by direct inspection.
\qed
\end{proof}

We can now give the proof of Theorem \ref{Deformation}.
\begin{proof}
Lemma \ref{Inequalitiies2} guarantees that $\mathfrak{M}_{1}(s\,|\,\alpha, \gamma, \rho_1),$ $\mathfrak{M}_{21}(s\,|\,\alpha, \gamma, \rho_{21}),$ and $\mathfrak{M}_{22}(s\,|\,\alpha, \gamma, \rho_{22})$ are analytic over the domain that is specified
in \eqref{mc} and \eqref{mc2}. Further, by Lemmas \ref{deformation21} and \ref{deformation22} we have the identities
\begin{align}
\mathfrak{M}_{1}(s\,|\,\alpha, \gamma, \rho_1)  & =  {\bf E}\Bigl[\beta_{22}^{s(1+\rho_1)}\Bigl(b_0=1+\tau-\rho_1s_0, \, b_1=\frac{2}{\gamma}(\alpha_1-\frac{\bar{\alpha}}{2}), \, b_2=\frac{2}{\gamma}(\alpha_3-\frac{\bar{\alpha}}{2})
\Bigr)\Bigr]   \nonumber \\
& \times 
{\bf E}\Bigl[\beta_{22}^{-\rho_1 s}\Bigl(b_0=1+\tau-\frac{2}{\gamma}\alpha_2, \, b_1=\frac{2}{\gamma}(\frac{\bar{\alpha}}{2}-\alpha_1), \, b_2=\frac{2}{\gamma}(\frac{\bar{\alpha}}{2}-\alpha_3\Bigr)\Bigr] 
, \\
\mathfrak{M}_{21}(s\,|\,\alpha, \gamma, \rho_{21})  & = {\bf E}\Bigl[\beta_{21}^{s(1+\rho_{21})}\Bigl(b_0=1+\tau-\frac{2}{\gamma}\alpha_1-\rho_{21}s_0, \, b_1=1+\tau+\frac{2}{\gamma}(\alpha_1-\frac{\bar{\alpha}}{2})\Bigr)\Bigr] 
\nonumber \\
& \times 
{\bf E}\Bigl[\beta_{21}^{-\rho_{21}s }\Bigl(b_0=1+\tau-\frac{2}{\gamma}\alpha_1, \, b_1=1+\tau+\frac{2}{\gamma}(\alpha_1-\frac{\bar{\alpha}}{2})\Bigr)\Bigr]
, \\
\mathfrak{M}_{22}(s\,|\,\alpha, \gamma, \rho_{22})  & = {\bf E}\Bigl[\beta_{21}^{s(1+\rho_{22})}\Bigl(b_0=1+\tau-\frac{2}{\gamma}\alpha_3-\rho_{22}s_0, \, b_1=\frac{2}{\gamma}(\frac{\bar{\alpha}}{2} - \alpha_1)\Bigr)\Bigr] \nonumber \\
& \times 
{\bf E}\Bigl[\beta_{21}^{-\rho_{22} s}\Bigl(b_0=1+\tau-\frac{2}{\gamma}\alpha_3, \, b_1=\frac{2}{\gamma}(\frac{\bar{\alpha}}{2} - \alpha_1)\Bigr)\Bigr],
\end{align}
and 
\begin{align}
\mathfrak{M}_{1}(s_0\,|\,\alpha, \gamma, \rho_1)  & = \mathfrak{M}_{1}(s_0\,|\,\alpha, \gamma) , \\
\mathfrak{M}_{21}(s_0\,|\,\alpha, \gamma, \rho_{21})  & = \mathfrak{M}_{21}(s_0\,|\,\alpha, \gamma) , \\
\mathfrak{M}_{22}(s_0\,|\,\alpha, \gamma, \rho_{22})  & = \mathfrak{M}_{22}(s_0\,|\,\alpha, \gamma),
\end{align}
where $\mathfrak{M}_{1}(s\,|\,\alpha, \gamma),$ $\mathfrak{M}_{21}(s\,|\,\alpha, \gamma),$ and
$\mathfrak{M}_{22}(s\,|\,\alpha, \gamma)$
are the factors of the original Mellin transform, cf. \eqref{MM1M21M22}. 
This proves Part I.

The symmetry in $(\alpha_1, \alpha_3)$ when $\rho_{21}=\rho_{22}$ derives from the fact that the
original factors $\mathfrak{M}_{1}(s\,|\,\alpha, \gamma)$ and $\mathfrak{M}_{2}(s\,|\,\alpha, \gamma),$ cf. \eqref{MM1M2},
are symmetric in $(\alpha_1, \alpha_3).$ The $\rho_1-$transformation preserves this symmetry in $\mathfrak{M}_{1}(s\,|\,\alpha, \gamma, \rho_1),$ cf. \eqref{aM1}. If $\rho_{21}=\rho_{22},$ then $\mathfrak{M}_{21}(s\,|\,\alpha, \gamma, \rho_{21})
\mathfrak{M}_{22}(s\,|\,\alpha, \gamma, \rho_{22})$ is also symmetric in $(\alpha_1, \alpha_3),$ cf.  \eqref{aM21} and
\eqref{aM22}. If $\rho_1=\rho_{21}=\rho_{22},$ then the deformed Mellin transform has the same symmetry as the
original Mellin transform, which is symmetric in $(\alpha_1, \alpha_2, \alpha_3).$ 

Finally, \eqref{atotalmass} follows from \eqref{shift} and the leading asymptotic of $\mathfrak{M}(s\,|\, \alpha,\gamma, \rho_1, \rho_{21}, \rho_{22})$ follows from (\ref{beta21asym}) and Stirling's formula. 
\qed
\end{proof}

We finally give the proof of Theorem \ref{Minimals0}.
\begin{proof}
The argument is very similar to the proof of Theorem \ref{Minimal}. The decomposition
\begin{equation}\label{MM1M2s0}
\mathfrak{M}'(s\,|\,\alpha, \gamma) = \Gamma(1+\frac{s}{\tau})  \,\mathfrak{M}_1(s\,|\,\alpha, \gamma) \,\mathfrak{M}_2'(s\,|\,\alpha, \gamma)
\end{equation}
in \eqref{MM1M2} still holds provided $\mathfrak{M}_1(s\,|\,\alpha, \gamma)$ is defined as before and $\mathfrak{M}_2'(s\,|\,\alpha, \gamma)$ is re-defined to be
\begin{align}
\mathfrak{M}_2'(s\,|\,\alpha, \gamma) =  &
\frac{\Gamma_2(s+1+\tau-\frac{2}{\gamma}\alpha_1\,|\,\tau)}{\Gamma_2(1+\tau-\frac{2}{\gamma}\alpha_1\,|\,\tau)}
\frac{\Gamma_2(2(1+\tau)-\frac{\bar{\alpha}}{\gamma}\,|\,\tau)}{\Gamma_2(s+2(1+\tau)-\frac{\bar{\alpha}}{\gamma}\,|\,\tau)}
\frac{\Gamma_2(1+\tau+\frac{2}{\gamma}(\alpha_2-\frac{\bar{\alpha}}{2})\,|\,\tau)}{\Gamma_2(s+1+\tau+\frac{2}{\gamma}(\alpha_2-\frac{\bar{\alpha}}{2})\,|\,\tau)}\nonumber \\
& \times
\frac{\Gamma_2(s+2(1+\tau)-\frac{2}{\gamma}\alpha_3\,|\,\tau)}{\Gamma_2(2(1+\tau)-\frac{2}{\gamma}\alpha_3\,|\,\tau)}.
\label{M2s0}
\end{align}
From the definition of $\beta_{22}$ we have,
\begin{equation}
\mathfrak{M}_2'(s\,|\,\alpha, \gamma) ={\bf E}\Bigl[\beta_{22}^{s}\Bigl(b_0=1+\tau-\frac{2}{\gamma}\alpha_1, \, b_1=1+\tau+\frac{2}{\gamma}(\alpha_1-\frac{\bar{\alpha}}{2}), 
b_2=  \frac{\bar{\alpha}}{\gamma} -  \frac{2\alpha_3}{\gamma} \Bigr)\Bigr].
\end{equation}
The probabilistic property of $\mathfrak{M}'(s\,|\,\alpha, \gamma)$ follows. As $s_0=0,$ the value of the pre-factor in \eqref{firstMs0} at $s=s_0$ is 1 so that
$\mathfrak{M}'(s_0\,|\,\alpha, \gamma)$ recovers the DOZZ formula. Finally, the leading asymptotic of $\mathfrak{M}'(s\,|\, \alpha,\gamma)$ follows from Stirling's formula.
\qed
\end{proof}

\section{Discussion and Open Questions}
In this section we will examine the limitations of our construction from the viewpoint of formulating a conjecture for the GMC on the sphere and present a comprehensive summary of all constraints that need to be satisfied to formulate such a conjecture. 

We start by examining the first moment  of the minimal solution, cf. \eqref{myfactor}. 
In the special case of $\alpha_i=0$ $M(\alpha,\gamma)$ has the Mellin transform given in \eqref{totalmass}
so that
our conjecture predicts
\begin{align}
{\bf E}[\rho_g(\alpha=0, \gamma)]=  %%& e^{2\frac{(1+\tau)}{\tau}\chi_g} \pi  \frac{\Gamma\big(\frac{1}{\tau}\bigr)\Gamma(\tau)}{\Gamma(1+\tau) \Gamma(1+\frac{1+\tau}{\tau})}, \\ = & 
e^{2\frac{(1+\tau)}{\tau}\chi_g} \frac{\pi}{1+\frac{1}{\tau}}
. 
\end{align}
Upon comparing with \eqref{firstmomentzeroalpha}, we observe that the extra factor of $1/(1+\frac{1}{\tau})$ makes this incompatible with
the correct first moment for any metric.

Next, we move on to the small deviation asymptotic of our construction. Given the known asymptotic of the Mellin transform of the general solution in
\eqref{mellinasymptotic}, we see that is satisfies \eqref{mellinasymptoticgeneral} so that our construction has the asymptotic,
\begin{equation}\label{oursmalldeviation}
{\bf P} (M(\alpha,\gamma)^{-1}\leq \varepsilon) \thicksim  e^{-\frac{1}{\varepsilon^{\frac{\tau}{2+\tau}}}}.
\end{equation}
While this asymptotic has the desired functional form, the exponent differs from the correct exponent $\varepsilon^\tau,$ cf. \eqref{smalldeviation}.
Incidentally, we note that the extra $1+\frac{1}{\tau}$ term in the exponent of the Mellin transform asymptotic in \eqref{mellinasymptotic}, which causes the wrong small deviation asymptotic, is the same as the extra factor in the first moment.

We now proceed to summarize our approach to the analytic continuation of the DOZZ formula in terms of new variables that allow one to succinctly
re-formulate the formula itself, all constraints, the analytic continuation, and challenges of finding other solutions. Define the variables
\begin{equation}
x_i = \frac{\bar{\alpha}}{\gamma} -  \frac{2\alpha_i}{\gamma}.
\end{equation}
Then, all the constraints following from $\tau>1,$ $\alpha_i$s satisfying \eqref{alphabarsum} and \eqref{bin1}, and $s_0$ in \eqref{s0} satisfying (\ref{mc}) and (\ref{mc2}) can be formulated in terms of the following basic constraints:
\begin{gather}
-\tau<s_0<\frac{1+\tau}{2}<\tau, \label{s0bounds} \\
s_0<x_i<s_0+1+\tau, \\
0<x_i<1+\tau-s_0, \\
x_1+x_2+x_3=s_0+1+\tau.
\end{gather}
Similarly, the constraints of $s$ satisfying (\ref{mc}) and (\ref{mc2}) %%along with the constraints on $\alpha_i$s
and $\tau>1$ imply:
\begin{gather}
s+\tau>0,  \\
s+\tau-s_0>0, \label{staus0bound} \\
s-s_0+x_i > 0, \\
s+1+\tau-x_i>0.
\end{gather}
These statements follow from Lemma \ref{Inequalitiies}. The gamma and double Gamma factors in the DOZZ formula then take on the form
\begin{equation}\label{DOZZxi}
\Gamma(1+\frac{s_0}{\tau})
\frac{\Gamma_2(s_0+1+\tau\,|\,\tau)}{\Gamma_2(1+\tau\,|\,\tau)}
\frac{\Gamma_2(1+\tau-s_0\,|\,\tau)}{\Gamma_2(1+\tau\,|\,\tau)} 
\prod\limits_{i=1}^3 
\frac{\Gamma_2(x_i\,|\,\tau)}{\Gamma_2(x_i-s_0\,|\,\tau)}
\frac{\Gamma_2(1+\tau-x_i\,|\,\tau)}{\Gamma_2(s_0+1+\tau-x_i\,|\,\tau)}
.
\end{equation}
The minimal solution is
\begin{equation}\label{minimalxi}
\Gamma(1+\frac{s}{\tau})
\frac{\Gamma_2(s+1+\tau\,|\,\tau)}{\Gamma_2(1+\tau\,|\,\tau)}
\frac{\Gamma_2(1+\tau-s_0\,|\,\tau)}{\Gamma_2(s-s_0+1+\tau\,|\,\tau)} 
\prod\limits_{i=1}^3 
\frac{\Gamma_2(s-s_0+x_i\,|\,\tau)}{\Gamma_2(x_i-s_0\,|\,\tau)}
\frac{\Gamma_2(1+\tau-x_i\,|\,\tau)}{\Gamma_2(s+1+\tau-x_i\,|\,\tau)}
.
\end{equation}
It is easy to see that the above bounds guarantee the positiveness of the arguments of all single and double Gamma factors and that this solution recovers the DOZZ formula at $s=s_0.$ The probabilistic structure of the minimal solution is 
\begin{equation}\label{minprobabxi}
 Y \,
\beta_{22}\bigl(b_0=1+\tau, \, b_1=-x_1, \, b_2=-x_3\bigr)
\beta_{21}\bigl(b_0=x_1-s_0, \, b_1=1+\tau-x_1\bigr)
\beta_{21}\bigl(b_0=x_3-s_0, \, b_1=x_1\bigr),
\end{equation} 
cf. Section 3 for the definitions of the random variables involved and Theorem \ref{Minimal} for the proof. Once again, the above bounds guarantee that $b_0>0,$
$b_1>0$ for all $\beta_{21}$ factors, and $b_1b_2>0,$ $b_0+b_1>0,$ $b_0+b_2>0,$ $b_0+b_1+b_2>0$ for the $\beta_{22}$ factor as required.

Now that we have summarized the key constraints of the problem and properties of our solution, we can consider the open question of what can be done differently.
\begin{itemize}
\item The purpose of the gamma factor $\Gamma(1+\frac{s_0}{\tau})$ in \eqref{DOZZxi} is to produce the $\Gamma_2(s_0\,|\,\tau)$ term as explained in
footnote \ref{myfoot}. One can look for analytic continuations of $\Gamma(1+\frac{s_0}{\tau})$ other than $\Gamma(1+\frac{s}{\tau}).$ 
%for example, $\Gamma(1+\frac{s_0}{\tau})/\Gamma(1+\frac{s-s_0}{\tau}),$ which is allowed by \eqref{staus0bound}.
Alternatively, one can replace the $\Gamma_2(s+1+\tau\,|\,\tau)$ term in \eqref{minimalxi} with $\Gamma_2(s+\tau\,|\,\tau),$ thereby removing the need to introduce the factor $\Gamma(1+\frac{s_0}{\tau})$ at all. While preserving the DOZZ formula at $s=s_0,$ this replacement however destroys the probabilistic structure of the solution 
%poses the challenge of proving that the resulting analytic continuation is the Mellin transform of a random variable 
as it affects $\mathfrak{M}_1(s\,|\,\alpha, \gamma),$  cf. \eqref{M1}, corresponding to the $\beta_{22}$ factor in \eqref{minprobabxi}
. %%The same challenge is encountered if one continues the term $\frac{\Gamma_2(1+\tau-s_0\,|\,\tau)}{\Gamma_2(1+\tau\,|\,\tau)} $ in an alternative form
%%\begin{equation}
%%\frac{\Gamma(1+\frac{s-s_0}{\tau})}{\Gamma(1-\frac{s_0}{\tau})}\frac{\Gamma_2(\tau-s_0\,|\,\tau)}{\Gamma_2(s-s_0+\tau\,|\,\tau)}.
%%\end{equation}
\item One can consider deformations of $\beta_{21}$ and $\beta_{22}$ that preserve the value of the Mellin transform at $s_0$ and are more general than those we considered in Lemmas \ref{deformation21} and \ref{deformation22}. For example, it follows from the  L\'evy-Khinchine representation of $\beta_{22},$ cf. Theorem 7.1 in \cite{Me18},
\begin{equation}
{\bf E}[\beta_{22}^{s_0}(b_0, b_1, b_2)] = {\bf E}\Bigl[\prod_i \beta_{22}^{\kappa_i s_0} (b_0^i, b_1, b_2)\Bigr]
\end{equation}
provided $\kappa_i>0,$ $b_0^i>0,$ and 
\begin{equation}
 (e^{-s_0 t}-1)e^{-b_0 t} = \sum_i (e^{-\kappa_i s_0 t}-1)e^{-b_0^i t}\, \forall t>0
.
\end{equation}
A similar result holds for $\beta_{21}.$ The main limitation of such deformations is that they do not change the small deviation asymptotic of the original random variable. 
%%\item
%%One can try to group the double gamma factors differently and, for example, try to extend $\frac{\Gamma_2(x_i\,|\,\tau)}{\Gamma_2(x_j-s_0\,|\,\tau)}$
%%or $\frac{\Gamma_2(x_i\,|\,\tau)}{\Gamma_2(s_0+1+\tau-x_i\,|\,\tau)}.$ The problem one encounters is that $s_0$ can be zero, cf. \eqref{s0bounds}, and
%%the value of the ratio at $s=0$ needs to be one, which leads to contradiction if $s_0=0.$ Thus, one is forced to group them as we did in \eqref{DOZZxi}, unless
%%one is willing to look for different solutions for $s_0<0,$ $s_0=0,$ and $s_0>0.$
\item We observe that the gamma factor $\Gamma(1+\frac{s}{\tau})$ alone produces the desired small deviation asymptotic in
\eqref{smalldeviation}. Thus, if it is indeed present in the ``correct'' solution, then the solution cannot have $\beta_{21}$ factors, cf. \eqref{mellinasymptoticgeneral}. In the minimal solution these factors arise from $\mathfrak{M}_2(s\,|\,\alpha, \gamma),$ cf. \eqref{M2},
\begin{align}
\mathfrak{M}_2(s\,|\,\alpha, \gamma) =  
\frac{\Gamma_2(s-s_0+x_1\,|\,\tau)}{\Gamma_2(x_1-s_0\,|\,\tau)}
\frac{\Gamma_2(1+\tau-s_0\,|\,\tau)}{\Gamma_2(s-s_0+1+\tau\,|\,\tau)}
\frac{\Gamma_2(1+\tau-x_2\,|\,\tau)}{\Gamma_2(s+1+\tau-x_2\,|\,\tau)}
\frac{\Gamma_2(s-s_0+x_3\,|\,\tau)}{\Gamma_2(x_3-s_0\,|\,\tau)},
\end{align}
and the two $\beta_{21}$ factors are identified in \eqref{minprobabxi}. This poses the question of whether one can find a deformation of $\mathfrak{M}_2(s\,|\,\alpha, \gamma)$ that is the Mellin transform of a $\beta_{22}$ instead. This is for example achieved by modifying the last ratio. The expression 
\begin{equation}
\frac{\Gamma_2(s-s_0+x_1\,|\,\tau)}{\Gamma_2(x_1-s_0\,|\,\tau)}
\frac{\Gamma_2(1+\tau-s_0\,|\,\tau)}{\Gamma_2(s-s_0+1+\tau\,|\,\tau)}
\frac{\Gamma_2(1+\tau-x_2\,|\,\tau)}{\Gamma_2(s+1+\tau-x_2\,|\,\tau)}
\frac{\Gamma_2(s-s_0+x_3+1+\tau\,|\,\tau)}{\Gamma_2(x_3-s_0+1+\tau\,|\,\tau)}
\end{equation}
is the Mellin transform of $\beta_{22}\bigl(b_0=x_1-s_0,\, b_1=1+\tau-x_1, \,b_2 = x_3\bigr).$ This however affects the value at $s_0$ and produces
$\frac{\Gamma_2(x_3+1+\tau\,|\,\tau)}{\Gamma_2(x_3-s_0+1+\tau\,|\,\tau)}$ instead of $\frac{\Gamma_2(x_3\,|\,\tau)}{\Gamma_2(x_3-s_0\,|\,\tau)}.$
In the special case of $s_0=0$ the two terms coincide, thereby producing a solution with the desired small deviation asymptotic, cf. Theorem \ref{Minimals0}.
\item
We observe that the conditions $s_0=0$ $(\bar{\alpha}=2Q)$ and $\alpha_k<Q,$ cf. \eqref{bin1} imply $\alpha_k>0$ $\forall k.$ Hence, to be able to compare what our construction in Theorem \ref{Minimals0} predicts for the first moment to the exact first moment of $\rho_g$ one needs to be able to compute the integral
\begin{equation}
\int\limits_{\mathbb{C}} 
g(x)^{1-\frac{\gamma}{4}\overline{\alpha}} |x|^{-\gamma\alpha_1}|1-x|^{-\gamma\alpha_2}
 dx,
\end{equation} 
cf. \eqref{exactfirstmoment}, for $\alpha_k>0$ $\forall k.$ This is an open problem for any metric.
\item
One can consider adding multiples of $s-s_0$ to the arguments of the double gamma factors in the minimal solution. This preserves the DOZZ formula but does not in general 
preserve the probabilistic property of the solution. For example, let us add $s-s_0$ to the terms $\Gamma_2(s+1+\tau-x_2\,|\,\tau)$ and $\Gamma_2(1+\tau-x_2\,|\,\tau)$
in \eqref{minimalxi} resulting in
\begin{align}
\Gamma(1+\frac{s}{\tau}) &
\frac{\Gamma_2(s+1+\tau\,|\,\tau)}{\Gamma_2(1+\tau\,|\,\tau)}
\frac{\Gamma_2(1+\tau-s_0\,|\,\tau)}{\Gamma_2(s-s_0+1+\tau\,|\,\tau)} 
\prod\limits_{i=1}^3 
\frac{\Gamma_2(s-s_0+x_i\,|\,\tau)}{\Gamma_2(x_i-s_0\,|\,\tau)} \times
\nonumber \\
& \times 
\prod\limits_{i=1, 3}
\frac{\Gamma_2(1+\tau-x_i\,|\,\tau)}{\Gamma_2(s+1+\tau-x_i\,|\,\tau)}
\frac{\Gamma_2(s-s_0+1+\tau-x_2\,|\,\tau)}{\Gamma_2(2s-s_0+1+\tau-x_2\,|\,\tau)}
.
\end{align}
One can show that the expression 
\begin{equation}
\Gamma(1+\frac{s}{\tau}) 
 \frac{\Gamma_2(s+1+\tau\,|\,\tau)}{\Gamma_2(1+\tau\,|\,\tau)}
 \prod\limits_{i=1, 3}
\frac{\Gamma_2(s-s_0+x_i\,|\,\tau)}{\Gamma_2(x_i-s_0\,|\,\tau)}
\frac{\Gamma_2(s-s_0+1+\tau-x_2\,|\,\tau)}{\Gamma_2(2s-s_0+1+\tau-x_2\,|\,\tau)}
\end{equation}
is the Mellin transform of a random variable of the type that appears in the 1D GMC measure on the interval, cf. Theorems 8.2 and 8.3 in \cite{Me18}. In particular, it has a lognormal factor. However, the remaining expression 
\begin{equation}
\frac{\Gamma_2(1+\tau-s_0\,|\,\tau)}{\Gamma_2(s-s_0+1+\tau\,|\,\tau)}\frac{\Gamma_2(s-s_0+x_2\,|\,\tau)}{\Gamma_2(x_2-s_0\,|\,\tau)}
 \prod\limits_{i=1, 3} \frac{\Gamma_2(1+\tau-x_i\,|\,\tau)}{\Gamma_2(s+1+\tau-x_i\,|\,\tau)}
\end{equation}
is not the Mellin transform of Barnes beta type random variables. This calculation suggests that the mechanism of a global lognormal factor that works in 1D might not work in 2D. It is an open question whether the DOZZ formula is compatible with a global lognormal factor.

%%We will give one example that also illustrates  
\end{itemize}

The constraints of the problem are very tight and pose significant challenges in constructing solutions that meet all the constraints and have the desired asymptotic behavior. The minimal solution is imposed by the structure of the constraints. The biggest open problem is in finding a degree of freedom that would allow to
control the asymptotic behavior without breaking the constraints. The asymmetric deformations of the minimal solution that we introduced provide several degrees of freedom but do not affect the small deviation asymptotic. In the special case of $s_0=0$ we produced a modification of the minimal solution that satisfies the constraints and the asymptotic.

\section{Conclusions}
We have constructed a family of probability distributions that have the property that their Mellin transform satisfies the DOZZ formula.
The family is parameterized by three deformation parameters. When they are all set to zero, we get the minimal solution, which is symmetric in all insertion points $(\alpha_1, \alpha_2, \alpha_3).$ By varying the deformation parameters, we can produce solutions
that are only symmetric in $(\alpha_1, \alpha_3)$ or not symmetric at all. The probability distributions are constructed from
products of independent Fyodorov-Bouchaud factor (Frechet) and powers of Barnes beta distributions of types $(2,1)$ and $(2,2).$ 
In the special case of all insertion points equal to zero, our distributions degenerate to a product of independent Frechet factors.

Our construction contains a continuous family of solutions that are symmetric in all insertion points $(\alpha_1, \alpha_2, \alpha_3).$ 
We have given a general construction of conformal metrics on the sphere that give rise to $\rho_g$ possessing this symmetry.
%%This poses the interesting problem of constructing conformal metrics on the sphere that give rise to $\rho_g$ possessing this symmetry.

We have advanced the theory of Barnes beta distributions by constructing one-parameter deformations of general Barnes beta
distributions of types $(2,1)$ and $(2,2)$ having the property that the deformed Mellin transform has the same value at a given point
as the original Mellin transform. 

Our construction provides the first necessary step towards formulating a conjecture for the law of the GMC on the Riemann sphere with three insertion points.
%%In particular, in the special case of zero insertion points our solution then corresponds to the total mass of the GMC. 
In support of our construction we note that in addition to satisfying the DOZZ formula, the Mellin transform of our family of distributions is analytic over the same domain as the known domain of analyticity of the Mellin transform of the GMC on the sphere. 
We also note that our construction is capable of reproducing the $(\alpha_1, \alpha_3)$ or $(\alpha_1, \alpha_2, \alpha_3)$ symmetries that are known to hold for the law of GMC corresponding to particular metrics. Our construction is nonetheless not a conjecture for the GMC as it does not match the first moment of the GMC law at zero insertion points $(\alpha_i=0 \,\forall i)$ for any metric or the known small deviation asymptotic of GMC laws corresponding to metrics with everywhere positive curvature. In the special case of $s_0=0$ $(\alpha_1+\alpha_2+\alpha_3=2Q)$ our modified construction satisfies this small deviation asymptotic making it a possible candidate for a conjecture. 

The task of formulating conjectures for specific metrics on the sphere is outside the scope of this paper as it requires going beyond the analytic continuation of the DOZZ formula alone and must await new advances on GMC, notably the understanding of how the GMC law depends on the choice of the conformal metric
and the computation of the first moment of the GMC law as a function of the insertion points for the specific metrics of interest. %Of particular interest is the open question of finding the correct analytic continuation of the metric-dependent scaling factor in the DOZZ formula. This is underscored by the fact that our construction is capable of matching the small deviation asymptotic only when $s_0=0$ when this scaling factor is identically one. 
Once this information becomes available, we expect that the metric-specific conjecture will be a refinement of our construction.

%we believe that the general solution is likely to be a deformation of our minimal solution. 

\section*{Acknowledgments}
The author gratefully acknowledges that the problem of constructing a probability distribution that satisfies the DOZZ formula 
was posed to the author by Vincent Vargas. The author wishes to thank the referees for many helpful comments.

\appendix
\section{Scaling and Symmetry of GMC on the Sphere}
This section is based on \cite{David}.
Let $X_g(s)$ be defined by \eqref{Xg}, the corresponding $\rho_g$ by \eqref{rhog}, and the metric-independent constant $C_\gamma(\alpha_1, \alpha_2, \alpha_3)$
be as in \eqref{Cconst} so that 
\begin{equation}
\Gamma(s) \, {\bf E}\bigl[ \rho_g(\alpha_1, \alpha_2, \alpha_3, \gamma)^{-s} \bigr]\Big\vert_{s=s_0} =  e^{\frac{s_0^2\gamma^2}{2}\chi_g}\,
C_\gamma(\alpha_1, \alpha_2, \alpha_3),
\end{equation}
%%where the constant $C_\gamma(\alpha_1, \alpha_2, \alpha_3)$ is independent of the metric $g(x).$
Using the scaling relationship, 
\begin{equation}
\chi_{\lambda g} = \chi_g + \frac{1}{2}\log\lambda,\label{lambdascaling}
\end{equation}
that holds for any constant $\lambda>0$ we can work out what $\rho_{\lambda g}(\alpha_1, \alpha_2, \alpha_3) $ is going to be.
Observing the identity in law
\begin{equation}
X_{\lambda g}(x) = X_g(x),
\end{equation}
then
\begin{equation}
M_{\gamma, g}(dx) = \lambda M_{ g}(dx).
\end{equation}
It follows from (\ref{rhog}) that
\begin{align}
\rho_{\lambda g}(\alpha_1, \alpha_2, \alpha_3)  = & \lambda^{1-\frac{\gamma}{4}\bar{\alpha}} \,   e^{\frac{\gamma^2 \log\lambda}{4}}  e^{\frac{\gamma^2 \chi_g}{2}}
 \int_\mathbb{C}
\frac{g(x)^{-\frac{\gamma}{4}\bar{\alpha}}}{|x|^{\gamma\alpha_1} |1-x|^{\gamma\alpha_2}} 
M_{\gamma, g}(dx), \nonumber \\
= &  \lambda^{1-\frac{\gamma}{4}\bar{\alpha}+\frac{\gamma^2}{4}}\, \rho_{g}(\alpha_1, \alpha_2, \alpha_3).\label{gscaling}
\end{align}
%%\textcolor{red}{Q3. Is this correct?} 
This is the fundamental scaling relationship that holds for any background metric.
Then, 
\begin{equation}\label{auxMetric}
\Gamma(s) \, {\bf E}\bigl[ \rho_{\lambda g}(\alpha_1, \alpha_2, \alpha_3)^{-s} \bigr]\Big\vert_{s=s_0} = 
\lambda^{-s_0(1-\frac{\gamma}{4}\bar{\alpha}+\frac{\gamma^2}{4})} \, e^{\frac{s_0^2\gamma^2}{2}\chi_g}\,
C_\gamma(\alpha_1, \alpha_2, \alpha_3).
\end{equation}
On the other hand,
\begin{align}
\Gamma(s) \, {\bf E}\bigl[ \rho_{\lambda g}(\alpha_1, \alpha_2, \alpha_3)^{-s} \bigr]\Big\vert_{s=s_0} = & e^{\frac{s_0^2\gamma^2}{2}\chi_{\lambda g}}\,
C_\gamma(\alpha_1, \alpha_2, \alpha_3), \nonumber \\
= & e^{\frac{s_0^2\gamma^2}{4}\log\lambda}\,
e^{\frac{s_0^2\gamma^2}{2}\chi_{g}}\,
C_\gamma(\alpha_1, \alpha_2, \alpha_3).\label{tmp}
\end{align}
Upon comparing (\ref{auxMetric}) and (\ref{tmp}), we get
\begin{equation}
-s_0(1-\frac{\gamma}{4}\bar{\alpha}+\frac{\gamma^2}{4}) = \frac{s_0^2\gamma^2}{4}.
\end{equation}
which is correct by the definition of $s_0.$

The importance of (\ref{gscaling}) is that it implies the following scaling invariance of the Mellin transform that holds for all $s$ where it is defined,
\begin{equation}
{\bf E}\bigl[ \rho_{\lambda g}(\alpha_1, \alpha_2, \alpha_3)^{-s} \bigr] = \lambda^{s s_0 \frac{\gamma^2}{4}} \,{\bf E}\bigl[ \rho_{g}(\alpha_1, \alpha_2, \alpha_3)^{-s} \bigr]. %%\label{rhogscaling}
\end{equation}

We end this section with a proof of symmetry of $\rho_g$ under  $\alpha_1\leftrightarrow \alpha_3$ for a class of metrics.
Assume that the background metric $g(x)$ satisfies the property 
\begin{equation}\label{inverseproperty}
g\bigl(\frac{1}{x}\bigr) = |x|^4\,g(x), \,x\in \mathbb{C}.
\end{equation}
Then, 
\begin{equation}
\rho_g(\alpha_3, \alpha_2, \alpha_1, \gamma) = \rho_g(\alpha_1, \alpha_2, \alpha_3, \gamma) 
\end{equation}
in law.

The proof is based on (\ref{Xg}). If (\ref{inverseproperty}) holds, then it is easy to check that the following identities hold,
\begin{align}
{\bf E}\Bigl[X_g\bigl(\frac{1}{x}\bigr)\, X_g\bigl(\frac{1}{y}\bigr)\Bigr] = &{\bf E}[X_g(x)\, X_g(y)] , \\
g\bigl(\frac{1}{x}\bigr) d^2 \bigl(\frac{1}{x}\bigr) = & g(x) d^2 x
.
\end{align}
The result now follows from (\ref{rhog}) by changing variables $x\rightarrow 1/x.$ 

The natural metrics on the Riemann sphere such as 
\begin{align}
g(x) = & |x|_+^{-4}, \\
g(x) = & \frac{1}{(1+|x|^2)^2},
\end{align}
satisfy (\ref{inverseproperty}) so that the corresponding  $\rho_g$ is symmetric under  $\alpha_1\leftrightarrow \alpha_3.$ %%We finally note that the constant $\chi_g$ can be computed explicitly. For example, for the metric  $g(x) = |x|_+^{-4}$ it equals 0. 

\section{A Construction of Symmetric GMC measures on the Sphere} 
In this section we will give an original construction of a class of metrics $g(x)$ having the property that the corresponding GMC measure $\rho_g(\alpha_1, \alpha_2, \alpha_3, \gamma) $ is symmetric in all $(\alpha_1, \alpha_2, \alpha_3).$
\begin{lemma}
Let $g(x)$ be an arbitrary conformal metric that satisfies \eqref{ginverseproperty}. Define the transformed metric $T[g],$
\begin{equation}
T[g](x) = g(x)  + \frac{g\bigl(\frac{x}{x-1}\bigr)}{|x-1|^4} + g(1-x)
.
\end{equation}
Then, for $x\in \mathbb{C},$
\begin{gather}
T[g](1-x)=T[g](x), \label{symmetry_refl}\\
T[g]\bigl(\frac{1}{x}\bigr) = |x|^4\,T[g](x). \label{symmetry_inv}
\end{gather}
\end{lemma}
It is elementary to see from \eqref{symmetry_refl} that 
\begin{align}
{\bf E}\Bigl[X_{T[g]}\bigl(1-x\bigr)\, X_{T[g]}\bigl(1-y\bigr)\Bigr] = &{\bf E}[X_{T[g]}(x)\, X_{T[g]}(y)] , \\
T[g]\bigl(1-x\bigr) d^2 \bigl(1-x\bigr) = & T[g](x) d^2 x
\end{align}
so that
\begin{equation}
\rho_g(\alpha_2, \alpha_1, \alpha_3, \gamma) = \rho_g(\alpha_1, \alpha_2, \alpha_3, \gamma) 
\end{equation}
in law. It is shown in Appendix A that \eqref{symmetry_inv} implies 
\begin{equation}
\rho_g(\alpha_3, \alpha_2, \alpha_1, \gamma) = \rho_g(\alpha_1, \alpha_2, \alpha_3, \gamma) 
\end{equation}
so that $\rho_{T[g]}$ is fully symmetric in all insertion points. 

\begin{proof}
First, due to \eqref{ginverseproperty} the original metric satisfies 
\begin{equation}
\lim\limits_{x\rightarrow \infty} g(x)\,|x|^4 = g(0)
\end{equation}
so that the point $x=1$ is a removable singularity of $T[g],$
\begin{align}
\lim\limits_{x\rightarrow 1} \Bigl[\frac{g\bigl(\frac{x}{x-1}\bigr)}{|x-1|^4}\Bigr] = &\Bigl[\lim\limits_{z\rightarrow \infty} g(z)\,|z-1|^4\Bigr], \\
= & g(0).
\end{align}
Next, consider how $T[g]$ transforms under $x\rightarrow 1-x.$
\begin{align}
T[g](1-x)=& g(1-x)  + \frac{g\bigl(\frac{1-x}{-x}\bigr)}{|x|^4} + g(x), \\
= &  g(x) +  \frac{g\bigl(\frac{x}{x-1}\bigr)}{|x-1|^4}   +   g(1-x),
\end{align}
where we applied \eqref{ginverseproperty} at the point $x/(x-1).$
Similarly, consider how $T[g]$ transforms under $x\rightarrow 1/x.$
\begin{align}
T[g](\frac{1}{x})=& g(\frac{1}{x})  + \frac{g\bigl(\frac{\frac{1}{x}}{\frac{1}{x}-1}\bigr)}{|\frac{1}{x}-1|^4} + g(1-\frac{1}{x}), \\
= & |x|^4\, g(x) + |x|^4\, \frac{g\bigl(\frac{1}{1-x}\bigr)}{|x-1|^4}   +  \Big|\frac{x}{x-1}\Big|^4 g(\frac{x}{x-1}), \\
= & |x|^4\Bigl(g(x)  + \frac{g\bigl(\frac{x}{x-1}\bigr)}{|x-1|^4} + g(1-x)\Bigr),
\end{align}
where we applied \eqref{ginverseproperty} at the points $x,$ $x/(x-1),$ and $1-x.$ \qed
\end{proof}

\end{document}